\documentclass[11pt,reqno]{amsart}
\usepackage[english]{babel}
\usepackage{amsfonts}
\usepackage[utf8]{inputenc}
\usepackage{textgreek}
\usepackage{amsthm}
\usepackage{mathrsfs}
\usepackage{amsmath}
\usepackage{mathrsfs}
\usepackage{enumerate}
\usepackage{cite}
\usepackage{bbm}
\usepackage{xcolor}
\usepackage{graphicx}
\usepackage{frontespizio}
\usepackage[makeroom]{cancel}
\usepackage[normalem]{ulem}
\usepackage{amssymb}
\usepackage[right=4cm,bottom=4cm,footskip=30pt,left=4cm]{geometry}
\usepackage[colorlinks, citecolor=citegreen, linkcolor=refred,urlcolor=blue, unicode,psdextra,hypertexnames=false]{hyperref}
\usepackage{mathtools}
\usepackage{enumitem}
\usepackage[toc,page]{appendix}
\usepackage{esint}
\usepackage{todo}
\usepackage{aligned-overset}
\definecolor{citegreen}{rgb}{0,0.3,0}
\definecolor{refred}{rgb}{0.5,0,0}
\makeatletter
\def\author@andify{
	\nxandlist {\unskip{} $\cdot$ \penalty-2}
	{\unskip {} $\cdot$ \penalty-2}
	{\unskip {} $\cdot$ \penalty-2}}
\makeatother

\title[Overview of regularity for the $p$-Laplacian in metric spaces]{An overview of regularity results for the Laplacian and $p$-Laplacian in metric spaces}
\let\oldemail\email
\def\email#1{\oldemail{\href{mailto:#1}{\textcolor{black}{#1}}}}
\author[I. Y.~Violo]{Ivan Yuri Violo}
\address{I. Y.~Violo,  Centro di Ricerca Matematica Ennio De Giorgi, Scuola Normale Superiore, Piazza dei Cavalieri 3, 56126 Pisa (PI), Italy}
\email{ivan.violo@sns.it}
\newcommand{\eps}{\varepsilon}

\newcommand{\rr}{\mathbb{R}}
\newcommand{\nn}{\mathbb{N}}

\newcommand{\nchi}{{\raise.3ex\hbox{$\chi$}}}

\newcommand{\sfd}{{\sf d}}
\newcommand{\Lip}{{\rm Lip}}
\newcommand{\id}{{\sf id}}
\renewcommand{\phi}{\varphi}
\newcommand{\restr}[1]{\lower3pt\hbox{$|_{#1}$}}

\newcommand{\X}{{\rm X}}

\definecolor{mygray}{gray}{0.9}
\newcommand{\diam}{{\rm diam}}

\newcommand{\la}{\langle}
\newcommand{\ra}{\rangle}
\renewcommand{\div}{{\rm div}}
\newcommand{\mea}{\mathfrak{m}}
\newcommand{\mm}{\mathfrak{m}}

\newcommand{\LIP}{\mathsf{LIP}}

\newcommand{\test}{{\sf{Test}}}

\renewcommand{\Cap}{{\rm Cap}}

\renewcommand{\d}{{\mathrm d}}
\newcommand{\loc}{\mathsf{loc}}
\newcommand{\W}{\mathit{W}^{1,2}}

\renewcommand{\H}[1]{{\rm Hess}(#1)}

\newcommand{\lip}{{\rm lip}}

\DeclareMathOperator*{\esssup}{ess\,sup}
\DeclareMathOperator*{\essinf}{ess\,inf}

\newcommand{\Xdm}{(\X,\sfd,\mm)}
\newcommand{\RCD}{\mathrm{RCD}}
\newcommand{\CD}{\mathrm{CD}}

\newcommand{\R}{\mathbb{R}}

\newcommand{\rcd}{\mathrm{RCD}}

\newcommand{\dom}{{\sf D}}

\newcommand{\vol}{{\rm Vol}}

\renewcommand{\Cap}{{\rm Cap}}

\renewcommand{\limsup}{\varlimsup}

\newcommand{\LL}{\mathcal{L}}
\newcommand{\alme}{\mea\text{-a.e.}}

\theoremstyle{plain}
\newtheorem{theorem}{Theorem}[section]
\newtheorem{lemma}[theorem]{Lemma}
\newtheorem{prop}[theorem]{Proposition}

\newtheorem{cor}[theorem]{Corollary}

\theoremstyle{definition}
\newtheorem{definition}[theorem]{Definition}
\newtheorem{remark}[theorem]{Remark}

\newtheorem{open}[theorem]{Open question}

\numberwithin{equation}{section}
\begin{document}
	\begin{abstract}
		We  review  some regularity results of the Laplacian and $p$-Laplacian in metric measure spaces. The focus is mainly on interior H\"older, Lipschitz and second-regularity estimates and on spaces supporting a Poincaré inequality or having  Ricci curvature bounded below. 
	\end{abstract}
	\maketitle
	\noindent MSC (2020): 
	35B65, 
        35J92,
	46E36, 
        58J05
	\medskip
	
	\noindent \underline{\smash{Keywords}}: nonlinear potential theory, metric  spaces, Laplacian, elliptic PDEs, regularity estimates, Ricci curvature, RCD spaces.

\tableofcontents
\section{Introduction}
In this note we will discuss some results about elliptic regularity in metric measure spaces.
We will be mainly concerned with equations of the type
\begin{equation}\label{eq:intro}
    \Delta_p u =f,
\end{equation}
where $\Delta_p$ is the $p$-Laplacian and $f$ is a given data in some $L^q$ space. It is by now well understood that if the underlying space is doubling and supports a Poincaré inequality then the classical De Giorgi-Nash-Moser theory applies and thus Harnack's inequalities and H\"older estimates are available. 
Nevertheless the H\"older regularity is the best one can obtain using solely the assumptions of doubling and Poincaré. Indeed there are examples of spaces satisfying these properties but where harmonic functions might not be Lipschitz. Furthermore the usual $L^2$-second order estimates are missing in this general setting, as the second order Sobolev space is not even clearly defined.  

In  recent years it has become more and more clear that to obtain second-order and the Lipschitz estimates for elliptic equations in metric spaces it is natural to assume that the Ricci curvature is bounded below in a weak sense. In the setting of Riemannian manifolds  the connection between Ricci curvature bounds and PDEs estimates  has been well known for a long time going back to the  gradient estimates  for harmonic functions by  Cheng and Yau  \cite{ChengYau75,yau1975harmonic}   and the Li-Yau inequality for the heat equation \cite{LiYau} (see also \cite{gidas1981global,WZ11,cogreen,SC92}  for other examples).

In the context of metric spaces the concept of Ricci curvature bounded below was introduced with \textit{Curvature Dimension (CD) condition}.  This framework was immediately demonstrated to be sufficiently powerful to derive fundamental geometric and functional inequalities like the Brunn-Minkowsi \cite{Sturm06-2}, (log)Sobolev \cite{Lott-Villani09,CM17},  isoperimetric \cite{CM17-Inv} and spectral \cite{Ketterer15}  inequalities. However,  only after the introduction of the stronger \textit{Riemannian Curvature Dimension (RCD) condition} condition it became possible to fully exploit  the geometric analysis machinery leading to  even stronger geometric properties and the development of more advanced calculus tools. Recall, for examples, the non-smooth splitting theorem  \cite{GiMar23,Gigli13over} and the rectifiability of RCD spaces \cite{Mondino-Naber14}.  For a detailed introduction to CD and RCD spaces and their properties we refer to the surveys \cite{AmbICM,gigli2023giorgi,SturmECM,Villsurvey} and references there in.

The advanced analytical tools in RCD spaces primarily arise from the possibility of integrating by parts, thereby allowing the definition of linear differential operators such as the Laplacian, divergence, and heat semi-group. Moreover, it is possible to give a meaningful notion of the second-order Sobolev space $W^{2,2}$ that contains many non-trivial functions as a consequence  of the following key inclusion
\begin{equation}\label{eq:intro l2}\tag{$\star$}
    \{u \in \W(\X) \ : \ \Delta u \in L^2(\mm)\}\subset W^{2,2}(\X)
\end{equation}
(see \cite{Gigli14}). For more on the calculus available in RCD spaces we refer to \cite{GP20,Gigli12}. Note that, a-posteriori, we can interpret \eqref{eq:intro l2} as an elliptic regularity result for the Laplacian operator with its own interest. In addition, Lipschitz estimates for harmonic functions and the Poisson equation have been established in RCD setting \cite{Kell17,Jiang} and, more recently, even for harmonic maps into CAT(0) spaces \cite{gigliharm,mondino2022lipschitz,boundaryharm}.
Despite these developments and the \eqref{eq:intro l2} being known for a long time, a comprehensive second-order elliptic regularity  theory is still missing. One of the main difficulties is the lack of a difference quotients method, which makes it complicated to treat operators other than the Laplacian.

Partly motivated by this, in collaboration  with L.\ Benatti \cite{BV24}, we recently studied the second-order regularity properties of the $p$-Laplacian operator in RCD spaces. Our main result, roughly speaking, is that
\begin{equation}\label{eq:intro lp} \tag{$\star_p$}
    \Delta_p u\in L^2(\mm) \implies |D u|^{p-2}\nabla u \in \W(\X),
\end{equation}
for bounded $\RCD$ spaces and for $p$ sufficiently close to two (see Theorem \ref{thm:main rcd inf} and Remark \ref{rmk:K negative}).  Additionally we managed to show that a suitable subclass of $p$-harmonic functions  are locally Lipschitz (see Theorem \ref{thm:lim reg electro}).

The goal of  this note is then twofold. The first  is to give a short  overview of the regularity results for the Laplacian and $p$-Laplacian  in metric measure spaces with focus on, but not limited to,  RCD spaces and on the aspects that we brtiefly described above. The second is to give an overview of the proof of \eqref{eq:intro lp}. 

The exposition will be organized as follows. In Section \ref{sec:pre} we fix some key notations and standing assumptions.  The main discussion starts in Section \ref{sec:def plap}  with the definition of the $p$-Laplacian via integration by parts. Then in Section  \ref{sec:boundary} we will present some  well established results for the existence and uniqueness of Dirichlet and Neumann boundary value problems on space supporting a Poincaré inequality. Adding also a locally doubling  assumption we overview in Section \ref{sec:holder} the validity of Harnack's and H\"older estimates. In Section \ref{sec:laplacian} we move to the setting of Ricci curvature bounded below. We will review the second order and Lipschitz regularity properties for the Laplacian,  present some recent results related to the regularity of eigenfunctions and  briefly discuss the problem of the unique continuation property. In Section \ref{sec:plap} we will present the main results of \cite{BV24} about the $p$-Laplacian and give an overview of the proof in Section \ref{sec:proof}. 
Along the exposition we will also state some open questions.

\bigskip

\textbf{Acknowledgments:}
The author is thankful to L.\ Benatti, N.\ De Ponti, S.\ Farinelli, N.\ Gigli,   J. Liu, F.\ Nobili, T.\ Rajala and S.\ Schulz for stimulating discussions around the topics of this note.

 This survey is based on  the talk  given by the author at the ``School and Conference on Metric Measure Spaces, Ricci Curvature, and Optimal Transport"   (MeRiOT 2024) in Varenna organized by F.\ Cavalletti, M.\ Erbar, J.\ Maas and K.T.\ Sturm, to whom we express our gratitude for this event.

\section{Standing assumptions and preliminaries}\label{sec:pre}
A metric measure space is a triple $\Xdm$, where $(\X,\sfd)$ is a complete and separable metric space and $\mm$ is a non-negative Borel measure. We  denote by $\LIP(\X)$ the space of Lipschitz function sin $\X$ and for all $\Omega\subset \X$ open we denote by $\LIP_{bs}(\Omega)$ the subset of $\LIP(\X)$ of functions having support contained in $\Omega.$  $\mathcal H^n$ stands for the $n$-dimensional Hausdorff measure on $(\X,\sfd).$ We also denote the local Lipschitz constant by
\[
\lip f (x)\coloneqq \limsup_{y \to x} \frac{|f(x)-f(y)|}{\sfd(x,y)},
\]
set to zero if $x$ is isolated.

We assume the reader to be familiar with the notion of Sobolev space $W^{1,p}(\X)$ over a metric measure space $\Xdm$. For the relevant background we refer the reader to the books \cite{GP20,BB,HKST15}. We limit ourselves to recall below a few facts and conventions that will be used through the note. 

For any $f \in W^{1,p}(\X)$ there is a meaningful notion of modulus of the gradient $|D f|_p\in L^p(\mm)$, called minimal weak upper gradient. For  any $\Omega\subset \X$ open the local Sobolev space $f \in W^{1,p}_\loc(\Omega)$  is the space of  functions $f\in L^p(\Omega)$  such that $f\eta\in W^{1,p}(\X)$ for all $\eta \in \LIP_{bs}(\Omega).$ 

In general $|D f|_p$ might depend highly on the exponent $p.$ In particular if $f \in W^{1,p}\cap W^{1,q}(\X)$ it can be that $|D f|_{p}\neq |D f|_q$ or, if $f \in W^{1,p}\cap L^q(\X)$ and $|D f|_p \in L^q(\X)$ it might be that $f \notin W^{1,q}(\X).$ Additionally the space $W^{1,2}(\X)$ is in general not Hilbert.

Nevertheless, throughout  this note and unless differently specified we will maintain the following standing assumptions:
\begin{enumerate}[label=\roman*)]
    \item  if $f \in W^{1,p}\cap W^{1,q}(\X) $, then $|D f|_p=|D f|_q$ $\mm$-a.e., moreover if $f\in W^{1,p}(\X)$ and $|D f|_p,f\in L^q(\X)$, then $f \in W^{1,q}(\X).$
    \item  \textit{$\Xdm$ is infinitesimally Hilbertian},\ i.e.\  the space $W^{1,2}(\X)$ is a Hilbert space.
\end{enumerate}
Condition i) is usually called  \textit{strong $p$-independency of weak upper gradients} (see \cite{GN22}). Condition ii) was instead introduced in \cite{Gigli09}. It was shown in \cite{Cheeger99} that i) holds for all locally doubling spaces  supporting a  local Poincaré inequality (see Definition \ref{def:doubling} and Definition \ref{def:poinc} respectively). Moreover  $\RCD(K,\infty)$ spaces  (see Definition \ref{def:rcd}) satisfy  ii) by very definition and   i)  thanks to \cite{GH14}.
Thanks to i) we will simply write $|D f|$ in place of $|D f|_p$. Moreover i) and ii) combined imply that, for all $p\in(1,\infty)$, the following pointwise scalar product is pointwise bilinear:
\[
W^{1,p}(\X)\times W^{1,p}(\X)\ni (f,g)\mapsto \la \nabla f, \nabla g\ra\coloneqq \frac12 \bigg(|D(f+g)|^2-|D f|^2-|D g|^2\bigg).
\]
In particular this implies that the weak upper gradient satisfies the usual parallelogram identity which implies that
\begin{equation}\label{eq:unif conv}
    \text{ $W^{1,p}(\X)$ is uniformly convex and thus reflexive for all $p\in(1,\infty)$,}
\end{equation}
 see \cite[Prop.\ 4.4]{GN22}  for the proof.

We finally recall briefly the notion of $L^2$-tangent bundle $L^2(T\X)$ introduced in \cite{Gigli14}, which has the structure of $L^2$-normed $L^\infty$-module endowed with a pointwise norm $|\cdot | : L^2(T\X) \to  L^2(\mm)$. This space comes with a linear gradient operator map $\nabla : \W(\X) \to L^2(T\X)$ which satisfies $|\nabla f|=|D f|$. We refer to \cite{GP20} for an introduction on normed-modules and tangent bundles on metric spaces.

\section{First order theory for the Laplacian and \texorpdfstring{$p$}{p}-Laplacian in metric setting}\label{sec:first}

\subsection{Variational approach and  integration by parts}\label{sec:def plap}
The most straightforward way to define solutions of an elliptic PDE in a metric measure spaces is by considering the related minimization problem.  The simplest case is the one of $p$-harmonic functions. 
\begin{definition}[$p$-harmonic functions]
    Let $\Xdm$ be a metric measure space, let $\Omega \subset \X$ be an open and fix $p\in(1,\infty)$. We say that $u \in W^{1,p}_\loc(\Omega)$ is  $p$-superharmonic (resp.\ $p$-subharmonic) in $\Omega$ provided
    \begin{equation}
        \int_\Omega |D u |_p^p\,\d \mm\le   \int_\Omega |D (u+\phi) |_p^p\,\d \mm, \quad \forall \phi \in \LIP_{bs}(\Omega), \quad \phi \ge0 \text{ (resp.\ $\phi \le 0$)}.
    \end{equation}
    If $u$ is both  $p$-superharmonic and $p$-subharmonic  we say that $u$ is $p$-harmonic in $\Omega.$
\end{definition}
Note that for the above definition no assumptions on $\Xdm$ are needed (we could even omit the standing assumptions i) and ii) stated in Section \ref{sec:pre}). 
For an account on the theory of $p$-harmonic function in metric spaces, without assumptions i) and ii) but locally doubling and supporting a Poincaré inequality,  we refer to \cite[Chapter 7]{BB} and references therein.

It is natural to ask whether $p$-harmonic functions satisfy the associated $p$-Laplace equation: $\Delta_p u=0$. To show this we need a suitable notion of $p$-Laplacian. 
\begin{definition}[$p$-Laplacian]
    Let $\Xdm$ be a metric measure space (satisfying  assumptions i) and ii) stated in Section \ref{sec:pre}), $\Omega\subset \X$ be open and $u\in W^{1,p}_\loc(\Omega)$. We say that $\Delta_p u\ge f$ (resp.\ $\Delta_p u\le f$ )  in $\Omega$ for some  $f\in L^1_\loc(\Omega)$ provided  
    \begin{equation}\label{eq:lapl def}
        -\int_\Omega |D u|^{p-2}\la \nabla u, \nabla \phi\ra \d \mm \ge \int_\Omega f \phi \d \mm, \quad \forall \phi \in \LIP_{bs}(\Omega), \quad \phi \ge0 \text{ (resp.\ $\phi \le 0$)}.
    \end{equation}
Moreover we say that $\Delta_p u= f$ if equality holds in \eqref{eq:lapl def}, in which case $f$ is unique and it is denoted by $\Delta_p u.$
\end{definition}
Since it can be that $p-2<0$, for convention we set $|D u|^{p-2}\la \nabla u, \nabla \phi\ra=0$ whenever $|Du|=0$. With this convention and the Cauchy–Schwarz inequality we have that $||D u|^{p-2}\la \nabla u, \nabla \phi\ra |\le |D u|^{p-1}\Lip(\phi)\in L^1_{\loc}(\Omega)$ and so the left hand side of \eqref{eq:lapl def} is well defined. For $p=2$ we will denote for simplicity $\Delta\coloneqq \Delta_2$. 

It is straightforward to check that the definition of $p$-Laplacian is compatible with the notion of $p$-harmonic functions.
\begin{prop}\label{prop:equivalence p}
    A function $u\in W^{1,p}_\loc(\Omega)$ is $p$-subharmonic (resp.\ $p$-superharmonic) in $\Omega$ if and only if $\Delta_p \ge 0$ (resp.\ $\Delta_p\le 0$) in $\Omega$. In particular $u$ is $p$-harmonic if and only if $\Delta_p u=0.$
\end{prop}
For a proof in the case $p=2$, which can be easily adapted to general $p,$ see \cite[Theorem 2.5]{GR19}.

It is worth to mention that it is possible to define a notion of $p$-Laplacian and $p$-harmonic functions via integration by parts using the Cheeger differentiability structure introduced in \cite{Cheeger99}. However this requires some further assumptions on the underlying space and the notions one would get are in general not equivalent to one stated above. Nevertheless much of the first order theory would still apply (e.g.\ all the results in Section \ref{sec:boundary} and Section \ref{sec:holder}). See \cite[Section 7.1 and Appendix B.2]{BB}  for a more detailed discussion on this topic.

\subsection{Boundary value problems}\label{sec:boundary}
In this section we give some examples of existence and uniqueness results for boundary value problems involving the $p$-Laplacian. This list, far from comprehensive, it is meant to show that it is possible to obtain non-trivial solutions to several type of PDEs linked to natural minimization problems under mild assumptions on the underlying space. All the results are essentially folklore and appeared often and in multiple variants in the literature.

The main tool to show existence is the standard direct method of the calculus of variations, which requires the underlying space to satisfy a Poincaré inequality. 
\begin{definition}[Local Poincar\'{e} inequality]\label{def:poinc}
    A metric measure space $\Xdm$ supports a \emph{local Poincar\'{e} inequality} provided  there exist a constant $\lambda\ge 1$, called \textit{dilation constant} such that
			for every $f \in \LIP_\loc(\X)$ it holds
			\[
			\fint_{B_r(x)}|f-f_{B_R(x)}|\,\d\mm\leq C_P({R})\,r\fint_{B_{\lambda r}(x)}\lip(f)\,\d\mm
			\quad \forall\,\, 0<r\le R, \,\, \forall\, \, x\in\X,
			\]
            where $f_{B_R(x)}\coloneqq \fint_{B_r(x)}f\,\d\mm$ and $C_{P}(R)>0$ is a constant depending only on $R.$ 
    \end{definition}
    
Throughout this section we will assume without further notice that $\Xdm$ supports  local Poincar\'{e} inequality (in addition to assumptions i) and ii) of Section \ref{sec:pre} which are always present in this note).

We start by considering Dirichlet boundary value problems.  First we recall that the above local Poincaré inequality implies the following classical Poincaré-type inequality  for functions that are zero at the boundary (see e.g.\ \cite[Corollary 5.54]{BB}).
\begin{prop}[Poincaré inequality with zero-Dirichlet boundary conditions]\label{prop:dir poinc}
 Let $\Omega\subset \X$ be open and bounded and such that $\mm(\X\setminus \Omega)>0$. Then for all $p\in(1,\infty)$ it holds
 \begin{equation}\label{eq:poinc dir}\tag{D}
     \int_\Omega |u|^pd \mm\le C_{p,\Omega} \int |D u|^p\d \mm, \quad \forall u \in W^{1,p}_0(\Omega),
 \end{equation}
 where $C_{p,\Omega}>0$ is a constant depending only on $\Omega$ and $p.$
\end{prop}

We can now state the first existence result. 
\begin{prop}[Poisson equation]
    Let $\Omega\subset \X$ be open and bounded with $\mm(\X\setminus \Omega)>0.$ Then for all $f \in L^{p'}(\Omega)$, $p'\coloneqq p/(p-1),$ and $g \in W^{1,p}(\Omega)$ there exists a unique solution to 
    \begin{equation}\label{eq:pepsi dir}
        \begin{cases}
            \Delta_{p} u=f,& \text{in $\Omega$},\\
            u \in g+W^{1,p}_0(\Omega).
        \end{cases}
    \end{equation}
     Moreover the solution to \eqref{eq:pepsi dir} coincides with the unique minimum of 
    \begin{equation}\label{eq:pepsi min}
        \inf \left\{F_p(u)\coloneqq \frac1p\int_\Omega |D u|^p-f u \d \mm \ : \
        u \in g+W^{1,p}_0(\Omega)
        \right\}. 
    \end{equation}
\end{prop}
\begin{proof}[Sketch of the proof] The existence of a minimizer $u\in g+W^{1,p}_0(\Omega)$ to \eqref{eq:pepsi min} is a straightforward application of the direct method in the calculus of variations. In particular using the Poincaré inequality in Proposition \ref{prop:dir poinc} and the reflexivity of $W^{1,p}(\X)$ (by \eqref{eq:unif conv}) we deduce the a minimizing sequence admits a $W^{1,p}$-weakly convergent subsequence. This combined with the lower semicontinuity of the $p$-energy under $W^{1,p}$-weak convergence (see \cite[Prop.\ 2.1.19]{GP20})  gives the existence of a minimizer. Uniqueness follows also easily from the uniform convexity of the $W^{1,p}$-norm (again by \eqref{eq:unif conv}). Finally the characterization of the minimizers \eqref{eq:pepsi dir} with the solutions of equation \eqref{eq:pepsi dir} can be easily deduced by using the convexity of $t\mapsto F_p(u+t\phi)$ for all $\phi \in \LIP_{bs}(\Omega)$ and differentiating at $t=0$ (cf.\ with Proposition \ref{prop:equivalence p}). 
\end{proof}
Minimizer of \eqref{eq:pepsi min} in metric setting were also analyzed in  \cite{GongManfrediMikko12}.

Next we present an example of Dirichlet eigenvalue boundary value problem, which can be dealt with as the previous one. For a proof of the uniqueness which covers this setting see \cite[Theorem 5.6]{MS20}.
\begin{prop}[Eigenvalue problem]
    Let $\Omega\subset \X$ be open and bounded with $\mm(\X\setminus \Omega)>0.$ Then  there exists a unique minimizer of
     \begin{equation}\label{eq:eigen min}
        \lambda_1(\Omega)\coloneqq \inf \left\{\frac{\int_\Omega {|D u|^p}  \d \mm}{\int_{\Omega}|u|^p\ d \mm} \ : \ u \in W^{1,p}_0(\Omega)\setminus\{0\}\right\},
    \end{equation}
    which satisfies
    \begin{equation}\label{eq:eigen dir}
            \Delta_{p} u=-\lambda_1(\Omega) u|u|^{p-2}, \quad \text{in $\Omega$}.
    \end{equation}
\end{prop}
 Minimizers of \eqref{eq:eigen min} in the metric setting were studied first in \cite{LMP06} deriving Harnack's inequalities, but never discussing the associated equation \eqref{eq:eigen dir}.

\begin{prop}[$p$-Electrostatic potential]\label{prop:electrop}
    Let $\Omega\subset \X$ be open and bounded with $\mm(\X\setminus \Omega)>0$ and let $K\subset \Omega$ be compact.
    Then  there exists a unique minimizer of
     \begin{equation}
        \Cap_p(K,\Omega)\coloneqq \inf \left \{\int_\Omega |D u|^p\d \mm \ : \ u \in W^{1,p}_0(\Omega),\, u\ge 1 \text{ in a neighborhood of $K$}\right\},
    \end{equation}
    called the $p$-electrostatic potential of $K$ in $\Omega$, which satisfies
     \begin{equation}
         \begin{cases}
             \Delta_p u=0, &\text{in $\Omega\setminus K$},\\
              u=1, & \text{in $ K$}.\\
         \end{cases}
    \end{equation}
\end{prop}	
The quantity $\Cap_p(K,\Omega)$, called \textit{$p$-capacity}, has been widely studied  and plays a fundamental role in theory of Sobolev spaces in metric setting (see \cite[Section 1.4]{BB}).

\begin{remark}
    The Poincaré inequality for functions with zero boundary condition in Proposition \ref{prop:dir poinc} (and thus all the previous existence results) holds under the weaker assumption that $\X\setminus \Omega$ has positive $p$-capacity (see \cite[Corolalry 5.54]{BB}). 
\end{remark}

\medskip

We pass now to  equations with  Neumann boundary conditions. 
\begin{definition}[$p$-Laplacian with Neumann boundary conditions]\label{def:neu}
   Fix $f \in L^1(\Omega)$. We say $u\in W^{1,p}(\Omega)$ satisfies 
    \begin{equation}
        \begin{cases}
         \Delta_p u=f, & \text{in $\Omega$}\\
         \partial_\nu u=0, &  \text{in $\partial \Omega$},
        \end{cases}
    \end{equation}
    provided
    \begin{equation}
        -\int_\Omega |D u|^{p-2}\la \nabla u, \nabla \phi\ra \d \mm = \int_\Omega f \phi \d \mm, \quad \text{ for all $\phi \in W^{1,p}\cap L^\infty(\Omega)$}.
    \end{equation}
\end{definition}
Note that expression `$ \partial_\nu u=0$ in $\partial \Omega$' is just formal as we are not defining the normal derivative.

To obtain existence of Neumann problem we replace \eqref{eq:poinc dir} with the following Poincaré inequality:
 \begin{equation}\label{eq:neu poinc}\tag{N}
     \int_\Omega |u-u_\Omega|^pd \mm\le C_{p,\Omega} \int |D u|^p\d \mm, \quad \forall u \in W^{1,p}(\Omega),
 \end{equation}
 where $u_\Omega\coloneqq \fint_\Omega u\d \mm$ and $C_{p,\Omega}>0$ is a constant depending only on $\Omega$ and $p.$
Sufficient conditions for inequality \eqref{eq:neu poinc} to hold are:
\begin{enumerate}
    \item $\Omega=\X$ with $\diam(\X)<\infty$ (see \cite[Theorem 2.4]{BV24});
    \item $\Xdm$ is locally doubling (see Definition \ref{def:doubling}) and $\Omega$ is a bounded uniform domain, that  is an open set connected by curves that are sufficiently away from the boundary (see \cite{unifpoinc}).
\end{enumerate}
Thanks to \cite{rajala2021approximation} we know that on doubling quasi-convex metric spaces that there are plenty of uniform domains.
 
The results stated above for the Dirichlet case hold also in the Neumann case. The argument is essentially the same, but using \eqref{eq:neu poinc} instead of \eqref{eq:poinc dir} to apply the direct method.
\begin{prop}[Poisson equation with Neumann boundary conditions]\label{prop:poisson neum}
    Suppose that $\Omega$ satisfies \eqref{eq:neu poinc} Then for all $f \in L^{p'}(\Omega)$ with $\int_{\Omega} f \d \mm=0$ and $p'\coloneqq p/(p-1)$ there exists a unique solution to 
    \begin{equation}
        \begin{cases}
            \Delta_{p} u=f,& \text{in $\Omega$},\\
             \partial_\nu u=0, &  \text{in $\partial \Omega$},\\
             \int_\Omega u \d\mm=0.
        \end{cases}
    \end{equation}
     Moreover this solution coincides with the unique minimum of 
    \begin{equation}
        \inf \left\{\int_\Omega \frac{|D u|^p}p-f u \d \mm \ : \ u \in W^{1,p}(\Omega),\, \int_\Omega u \d \mm=0\right\}.
    \end{equation}
\end{prop}

\begin{prop}[Neumann eigenvalue problem]
       Suppose that $\Omega$ satisfies \eqref{eq:neu poinc}. Then  there exists a unique minimizer of
     \begin{equation}
          \lambda_1^N(\Omega)\coloneqq \inf \left\{\frac{\int_\Omega {|D u|^p}  \d \mm}{\int_{\Omega}|u|^p\ d \mm} \ : \ u \in W^{1,p}(\Omega)\setminus\{0\},\,\, \int_\Omega u|u|^{p-2} \d \mm=0\right\},
    \end{equation}
    which satisfies
    \begin{equation}
    \begin{cases}
          \Delta_{p} u=-\lambda_1^N(\Omega) u|u|^{p-2}, & \text{in $\Omega$}\\
          \partial_\nu u=0, &  \text{in $\partial \Omega$.}
    \end{cases}
    \end{equation}
\end{prop}

\subsection{Interior and boundary regularity of solutions}\label{sec:holder}
A standard method in the Euclidean setting to obtain interior H\"older continuity estimates for elliptic equations in divergence form is the Moser's iteration method. Ultimately this boils down to several applications of the Sobolev inequality to get better and better integrability estimates (see e.g.\ \cite[Section 8.5]{GilbargTrudinger}).  It is now well understood that a \textit{doubling assumption} combined with a \textit{Poincaré inequality} is sufficient to obtain a \textit{Sobolev inequality}. This observation goes back to  \cite{Saloff92}  and was generalized to the metric setting in \cite{HK00}.

 Let us start recalling the notion of a locally doubling metric measure space.
\begin{definition}[Locally doubling condition]\label{def:doubling}
    	$\Xdm$ is said \emph{locally doubling} if for all $R>0$ there exists a constant   $C_{D}(R)>0$ such that
			\[
			\mm\big(B_{2r}(x)\big)\leq C_D(R)\,\mm\big(B_r(x)\big)\quad\text{ for every }0<r<R\text{ and }x\in\X.
			\]
\end{definition}
 To state the Sobolev inequality  we need a notion of `dimension', which allows to define the Sobolev conjugate exponent $p^*.$
\begin{prop}[Doubling dimension]
 Suppose $\Xdm$ is locally doubling. Then there exists a constant $s>1$ and for all $R>0$ there exists a constant $c_{R}>0$ such that 
 \begin{equation}\label{eq:d dim}
     \frac{\mm(B_{r_1}(x))}{\mm(B_{r_2}(x))}\ge  c_R\left(\frac{r_1}{r_2}\right)^s, \quad \text{for every $0<r_1<r_2<R$}.
 \end{equation}
\end{prop}
Note that if \eqref{eq:d dim} holds for some $s$ then it holds for all $s'>s$ and so in a sense $s$ might be understood as an  `upper bound on the dimension' of $\Xdm.$
\begin{theorem}[{Local Sobolev inequality, \cite[Theorem 5.1]{HK00}}]
   Let $\Xdm$ be  locally doubling and supporting a local Poincaré inequality. Fix $s>1$  satisfying \eqref{eq:d dim}, fix $p\in(1,s)$ and set $p^*\coloneqq \frac{ps}{s-p}$. Then for all $B_r(x)\subset \X$, $r\le R$ it holds
    \begin{equation}
        \left(\fint_{B_r(x)} |u|^{p^*}\d \mm\right)^\frac{p}{p^*}\le C \fint_{B_{2\lambda r}(x)} r^p|D u|^p+|u|^p\d \mm, \quad \forall u \in W^{1,p}_{\loc}(\X),
    \end{equation}
    where $C$ is a constant depending only on $p,\lambda,C_D,C_P,s$. The same holds for $p=n$ and with $p^*$ arbitrarily chosen. 
\end{theorem}
In the case $p>s$ we have also a counterpart of the Morrey's embedding.
\begin{theorem}[{Morrey's embedding, \cite[Theorem 5.1]{HK00}}]\label{thm:morreyà}
      Let $\Xdm$ be  locally doubling and supporting a local Poincaré inequality. Fix $s>1$  satisfying \eqref{eq:d dim} and let $p>s$. Then any function $u \in W^{1,p}_{\loc}(\X)$ has a locally H\"older continuous representative. 
\end{theorem}


From now on throughout this subsection we will assume  that $\Xdm$ is locally doubling and supports a local Poincaré inequality. We also fix $s>1$  satisfying \eqref{eq:d dim}.
We can now state a version of the Harnack's inequalities.
\begin{theorem}[Harnack's inequality for subsolutions]\label{thm:sup harnack}
    Fix $p\in(1,s]$ and $q>s/p.$ Suppose that 
    $$\Delta_pu\ge f+g|u|^{p-1} \quad \text{in $B_1(x)\subset  \X$,} \quad    \mm(B_1(x))=1,$$
    for some $f,g\in L^q(B_1(x))$, with $\|g\|_{L^q(B_1(x))}\le c_0.$ 
    Then  it holds
		\begin{equation}\label{eq:sup estimate}
			\esssup_{B_{\frac 12}(x)} u \le C_1 \left(\|u^+\|_{L^1(B_1(x))}+\|f\|_{L^q(B_1(x))}^\frac1{p-1}\right),
		\end{equation}
			where  $C_1$ is a constant depending only on $p,q,\lambda,C_D,C_P,s,c_0$. 
\end{theorem}

\begin{theorem}[Harnack's inequality for supersolutions]\label{thm:sub harnack}
    Fix $p\in(1,s]$ and $q>s/p.$ Suppose that $u$ is non-negative and satisfies
    $$\Delta_pu\le f+g|u|^{p-1} \quad \text{in $B_{35\lambda}(x)\subset  \X$,} \quad    \mm(B_1(x))=1,$$
    for some $f,g\in L^q(B_{35\lambda}(x))$, with $\|g\|_{L^q(B_{35\lambda}(x))}\le c_0.$ 
    Then  there exists $m\in[1,\infty)$ such that
		\begin{equation}\label{eq:sub estimate}
			\essinf_{B_{\frac 12}(x)} u  + \|f\|_{L^q(B_{25}(x))}^\frac{1}{p-1}  \ge C_2 \|u\|_{L^m(B_1(x))}
		\end{equation}
			where  $C_2$ and $m$ depend only on $p,q,\lambda,C_D,C_P,s,c_0$. 
\end{theorem}
Harnack's inequalities in metric setting for $p$-sub/superharmonic functions were first obtained in \cite{KS01} using De Giorgi's method and later in \cite{BM06} using Moser's iteration. In \cite{LMP06} the same was shown for $p$-eigenfunctions and in \cite{GongManfrediMikko12} for the $p$-Poisson equation. These results did not use the PDE-formulation, in fact they hold without the standing assumptions i) and ii) of Section \ref{sec:pre}. We refer to \cite[Chapter 8]{BB} for more detailed references and self-contained proofs. On the other hand the Harnack's inequalities in the PDE-formulation, as stated in Theorem \ref{thm:sup harnack} and Theorem \ref{thm:sub harnack}, are considered folklore among experts as the argument is ultimately the same as in the Euclidean space. A full self contained proof can be found in the appendix of \cite{BV24}.

We next present some standard consequences of the  Harnack's inequalities.
\begin{cor}[H\"older regularity of solutions]\label{cor:holder}
     Fix $p\in(1,s]$ and $q>s/p.$  Suppose that 
    $$\Delta_pu=f+gu|u|^{p-2} \quad \text{in $\Omega\subset \X$},$$
    for some $f,g\in L^q_{\loc}(\Omega)$. Then $u$ has a locally H\"older continuous representative in $\Omega.$
\end{cor}
The proof of the above starting from inequalities \eqref{eq:sup estimate} and \eqref{eq:sub estimate} is the same as in the Euclidean setting (see e.g.\ \cite[Section 4.4]{HanLinbook})
\begin{remark}
    Note that if $p>s$ the solution is automatically H\"older continuous by the Morrey's embedding of Theorem \ref{thm:morreyà}.
\end{remark}

\begin{cor}
    $p$-superharmonic (resp.\ $p$-subharmonic) functions have a lower (resp.\ upper) semicontinuous representative.
\end{cor}
\begin{proof}
Suppose $u$ is $p$-superharmonic in $\Omega$. For all $x \in \Omega$ set \sloppy $u^*(x)\coloneqq \lim_{r \to 0^+} \essinf_{B_r(x)} u $. It is easy to check that $u^*$ is lower semicontinuous. By the Lebesgue differentiation theorem (which holds under the locally doubling  assumption, see e.g.\ \cite[Section 3.4]{HKST15}) we have that $\fint_{B_r(x)} |u-u(x)|\d \mm \to 0$ for $\mm$-a.e.\ point $x\in \Omega$, called Lebesgue point. Clearly by definition $u^*(x)\le \limsup_{r \to 0^+} \fint_{B_r(x)}u\d \mm=u(x)$ for each Lebesgue point $x.$ It is not hard to check that $(u(x)-u)^+$ is $p$-subharmonic in $\Omega$ (see e.g.\ \cite[Prop.\ 7.12]{BB}). Hence by Harnack's inequality in \eqref{eq:sup estimate} 
\[
\esssup_{B_r(x)} (u(x)-u)\le \esssup_{B_r(x)} (u(x)-u)^+ \le C \fint_{B_r(x)}|u-u(x)|\d \mm \to 0,
\]
which implies that $u^*\ge u$ $\mm$-a.e.\ in $\Omega$ as well.
\end{proof}

\begin{cor}[Strong maximum principle]
Let $\Omega \subset \X$ be connected. Suppose that $u$ is a $p$-superharmonic and lower semicontinuous function in $\Omega$ which attains its minimum in $\Omega$. Then $u$ is constant. 
\end{cor}
\begin{proof}
    We can assume that $\min_{\Omega} u=0$. Then the set $C\coloneqq \{ x \in \Omega \ : \  u(x)=0\}$ is closed. For all $x \in C$ we have $B_{35 r\lambda }(x)\subset \Omega$ for $r$ small enough. The rescaled version of \eqref{eq:sub estimate} implies that $ \|u\|_{L^m(B_r(x))}=0$ and so $u\equiv 0$ in $B_r(x).$ Hence $C$ is also open and thus $C=\Omega.$
\end{proof}

The H\"older regularity result in Corollary \ref{cor:holder} is sharp in the sense that there are examples of harmonic functions in locally doubling metric measure spaces supporting a Poincaré inequality that are not locally Lipschitz (see the introduction of \cite{Koskela-Rajala-Shanmugalingam03}).

We conclude this part by mentioning some results about continuity at the boundary for $p$-harmonic functions in metric setting. Roughly speaking the key assumption is that the complement of the open set is sufficiently big. Recall the following standard definition.
\begin{definition}[Corkscrew condition]\label{def:cork}
    A set $\Omega \subset \X$ satisfies the \textit{interior (resp.\ exterior) corkscrew condition} if there exists constants $\lambda>0$ and $r>0$ such that for all $x\in \partial \Omega$ and $s\in(0,r)$ there exists a ball of radius $\lambda s$ contained in $B_s(x)\cap \Omega$ (resp.\ $B_s(x)\cap (\X\setminus E)$).
\end{definition}
Notable examples of sets satisfying the interior (resp.\ exterior)  corkscrew condition  are enlargements of sets (reps.\ the complement of enlargements of sets) in length spaces. In particular balls satisfy the interior corkscrew condition.
\begin{prop}[Continuity up to boundary of $p$-harmonic functions]\label{prop:boundary cont}
     Fix $p\in(1,\infty)$. Let $\Xdm$ be locally doubling and supporting a local Poincaré inequality. Let $\Omega\subset \X$ be bounded and satisfying the exterior corkscrew condition and $\mm(\X\setminus \Omega)>0$. Let also $f:\partial \Omega\to \rr$ be continuous. Then there exists a unique $u \in C(\overline\Omega)$ such that
    \[
    \begin{cases}
        \Delta_pu=0, & \text{in $\Omega$},\\
         u=f, & \text{in $\partial \Omega$}.
    \end{cases}
    \]
\end{prop}
We can also obtain the regularity up the boundary for the $p$-electrostatic potential introduced in Proposition \ref{prop:electrop}.
\begin{prop}[Boundary continuity of $p$-electrostatic potential]\label{prop:improved prop}
   Fix $p\in(1,\infty)$. Let $\Xdm$ be locally doubling and supporting a local Poincaré inequality. Let $\Omega\subset \X$ be bounded and satisfying the exterior corkscrew condition and $\mm(\X\setminus \Omega)>0$. Let also $K\subset \Omega$ be compact and satisfying the interior corkscrew condition. Then the $p$-electrostatic potential of $K$ in $\Omega$ (as in Proposition \ref{prop:electrop}) has a continuous representative in $\overline \Omega$ and satisfies
   \[
   \begin{cases}
       u=0, & \text{in $\partial \Omega$},.\\
       u=1, & \text{in $\partial K$}.\\
   \end{cases}
   \]
\end{prop}
    The proof of both Proposition \ref{prop:boundary cont} and Proposition \ref{prop:improved prop} can be found in \cite[Corollary 11.25]{BB} (see also \cite[Appendix B]{GV23}).

In fact much weaker conditions than the corkscrew condition are sufficient  to ensure continuity of the solution up to the boundary, such as the generalization of the classical Wiener criterion to metric setting (see \cite{Bjwiener,bjornwiener2} and also  \cite[Section 11.4]{BB}).

\section{Laplacian under Ricci lower bounds}\label{sec:laplacian}
\subsection{Second order regularity  and the definition of RCD space}
A standard fact about elliptic regularity is that given any open set $\Omega\subset \rr^n$ and a function $f \in \W(\Omega)$ which satisfies 
\begin{equation}
    \Delta u=f \in L^2(\Omega),
\end{equation}
it holds that $u\in W^{2,2}(\Omega')$ for all $\Omega'\subset \subset \Omega$ (see e.g.\ \cite[Section 6.3.1]{evans98}).
The motivation behind this result, at least formally, is the following identity
\begin{equation}\label{eq:eucl bochner}
   \frac 12 \Delta|D u|^2= |\H u|_{HS}^2+\la \nabla \Delta u,\nabla u\ra \quad \forall u \in C^\infty_c(\rr^n).
\end{equation}
Indeed, integrating \eqref{eq:eucl bochner} over all $\rr^n$ and integrating by parts the term on the left hand side vanishes and we obtain
\begin{equation}
    \int_{\rr^n}  |\H u|_{HS}^2 \le \int_{\rr^n} (\Delta u)^2.
\end{equation}
If instead of the flat Euclidean space we consider a general Riemannian manifold $(M,g)$ we can not expect \eqref{eq:eucl bochner} to hold. Indeed to obtain \eqref{eq:eucl bochner} we need to exchange the order of derivatives and thus expect some curvature terms to appear. As a matter of fact the correct analog of \eqref{eq:eucl bochner} is the \textit{Bochner-Weitzenb\"ock identity}:
\begin{equation}\label{eq:smooth bochner}
    \frac 12 \Delta_g|\nabla u|^2= |\H u|_{HS}^2+g (\nabla \Delta_g u,\nabla u) + {\rm Ric}_g(\nabla u,\nabla u), \quad \forall u \in C^\infty_c(M),
\end{equation}
where $\Delta_g$ is the Laplace-Beltrami operator and ${\rm Ric}_g$ is the Ricci tensor. We omit the definition of the Ricci tensor, but we observe that one can take \eqref{eq:smooth bochner} as characterization of it. From \eqref{eq:smooth bochner} it is apparent that in order to be able to estimate the Hessian in terms of the Laplacian we need to assume that the Ricci tensor is bounded from below. More precisely, assuming that ${\rm Ric}_g\ge Kg$ for some $K\in \R$ and integrating by parts \eqref{eq:smooth bochner} gives
\begin{equation}\label{eq:2calderon smooth}
    \int_{M}  |\H u|_{HS}^2\, \d \vol_g \le \int_{M} (\Delta u)^2-K |\nabla u|^2 \, \d \vol_g.
\end{equation}
All in all we see that a lower bound in the Ricci curvature entails second-order elliptic regularity. Therefore it is natural to try to force such condition on a metric measure space to get $W^{2,2}$-regularity from the Laplacian. There are however two obvious obstacles: how to impose Ricci curvature on a non-smooth structure and what is the $W^{2,2}$-space and the Hessian in metric setting. Roughly speaking, the answer to both of these issues is in identity \eqref{eq:smooth bochner}. We proceed in two steps. Multiplying \eqref{eq:smooth bochner} by a non-negative test function $\phi \in C^\infty_c(M) $ and neglecting the non-negative Hessian term we obtain the following weaker inequality:
\begin{equation}\label{eq:weak bochner}
     \frac 12 \Delta_g|D u|^2\ge g (\nabla \Delta_g u,\nabla u) + K |D u|^2, \quad \forall u \in C^\infty_c(M),
\end{equation}
The key observation is that we can give a meaning to \eqref{eq:weak bochner} also in metric measure space $\Xdm$, by using the language of distributional Laplacian introduced in Section \ref{sec:def plap}.  More precisely it makes sense to write
\begin{equation}\label{eq:weak bochner mms}\tag{B}
            \int_\X |D u|^2  \Delta h\,\d \mm \ge \int_\X \la \nabla u, \nabla \Delta u\ra h + K|D u|^2 h\d \mm,
\end{equation}
for all functions $u,h \in \W(\X)$ such that $\Delta u\in \W(\X)$, $h\ge 0$ and $h,\Delta h \in L^\infty(\mm)$. Inequality \eqref{eq:weak bochner mms} is called \textit{weak Bochner inequality} and can be used to define a notion a weak notion of Ricci curvature bonded below by $K$ in metric setting.
\begin{definition}[\hspace{1pt}\cite{AmbrosioGigliSavare12}]\label{def:rcd}
    Fix $K\in \rr.$ We say that inf.\ Hilbertian m.m.s.\  $\Xdm$ is an $\RCD(K,\infty)$ provided:
    \begin{enumerate}
        \item For all for all functions $u,h \in \W(\X)$ such that $\Delta u\in \W(\X)$ and $h,\Delta h \in L^\infty(\mm)$  the weak Bochner inequality \eqref{eq:weak bochner mms} holds.
        \item If $u \in \W(\X)$ and $|D u|\le 1$, then $u$ has a 1-Lipschitz representative.
        \item There exists $C>0$ and $x_0\in \X$ such that $\mm(B_R(x_0))\le e^{CR^2}$ holds for all $R>0.$
    \end{enumerate}
\end{definition}
\begin{remark}[Existence of many test functions]
   Is important to stress that there are many non-trivial test functions $u,h$ as in (1) in Definition \ref{def:rcd}. Indeed if $f \in L^2(\mm)$, then denoting by  $t\mapsto h_t f$ its evolution via the heat flow, it holds that $h_t f \in \W(\X)$ for $t>0$ and by linearity $\Delta h_t f= h_t \Delta f$, thus also $\Delta h_t f \in \W(\X)$. With a further regularization it is possible to obtain also a bounded Laplacian. For the definition of the heat flow and the proofs of these properties we refer to \cite{GP20}.
\end{remark}
The above definition in fact came only after the one  of $\CD(K,\infty)$ spaces introduced with a completely different approach via optimal transport in \cite{Sturm06-1} and \cite{Lott-Villani09}.  In fact $\RCD(K,\infty)$ spaces are equivalent to $\CD(K,\infty)$ spaces coupled with the infinitesimally Hilbertian condition \cite{AmbrosioGigliSavare12}.  We refer to the surveys \cite{AmbICM,gigli2023giorgi} for a background and an historical account on the topic of weak Ricci curvature bounds in metric setting. We also recall that $\RCD(K,\infty)$ spaces support a local Poincaré inequality (see \cite{Rajala12,rajala2012interpolated}) hence the results of Section \ref{sec:boundary} apply.

Going back to elliptic regularity we can now wonder if the weaker \eqref{eq:weak bochner} is still enough to obtain $W^{2,2}$-regularity (whatever that means), as we removed the Hessian term. Multiplying \eqref{eq:weak bochner} by $|\nabla u|^2$, integrating and then integrating by parts the left hand side we obtain
\begin{equation}
    \int |\nabla |\nabla u|^2|^2\d \vol_g\le 2\Lip(u)^2\int |\nabla u||\nabla \Delta u|+ K|\nabla u|^2 \d \vol_g,
\end{equation}
which formally entails the following  regularity result:
\begin{equation}\label{eq:weaker regularity}
    \Delta u \in \W(M),\, u \in \Lip(M) \implies |D u|^2 \in \W(M).
\end{equation}
Note that  $u \in W^{2,2}\cap \LIP(M)\implies |D u|^2 \in \W(M)$, hence \eqref{eq:weaker regularity} is in a sense a weaker version of \eqref{eq:2calderon smooth}.

It turns out that the same argument can be employed in RCD metric measure spaces.
\begin{theorem}[\hspace{1pt}\cite{Savare13}]
    Let $\Xdm$ be an $\RCD(K,\infty)$ space with $K\in \rr.$ Let $u \in \W\cap \LIP(\X)$ be such that $\Delta u \in \W(\X) $. Then $|D u|^2\in \W(\X).$
\end{theorem}
The above regularity result motivates the definition of the space of test functions:
\[
\test(\X)\coloneqq \left\{f \in \W\cap \LIP\cap L^\infty(\X) : \  \Delta u \in \W(\X) \right\}.
\]

A remarkable observation due to Bakry  \cite{Bakry83} in the context of abstract $\Gamma_2$-calculus  is the so-called self-improvement of the Bochner inequality. Roughly speaking it turns out that the weaker inequality in  \eqref{eq:weak bochner} implies the stronger version containing the  extra Hessian term. This can be shown by plugging in \eqref{eq:weak bochner} functions of the type $u=\Phi(f)$, where $\Phi$ are suitable polynomials and by carefully expanding all the terms via the chain and Leibniz rules.
These ideas were extended to $\RCD(K,\infty)$ spaces in \cite{Savare13} and \cite{Gigli14} leading to the following pivotal result. 
\begin{theorem}[\hspace{1pt}{\cite[Theorem 3.3.8]{Gigli14}}]\label{thm:improved bochner mms}
    Let $\Xdm$ be an $\RCD(K,\infty)$ space, $K\in \rr$, and  $u\in \test(\X)$. Then  $u \in  W^{2,2}(\X)$ and 
    \begin{equation}\label{eq:improved bochner mms}
        -\int_\X \frac12\la\nabla |D u|^2,\nabla \phi\ra \d \mm \ge \int_\X  |\H u|^2 \phi+ \la \nabla u, \nabla \Delta u\ra \phi + K|D u|^2 \phi\d\mm
    \end{equation}
   holds  for every $\phi \in L^\infty\cap \W(\X)$.
\end{theorem}
The definitions of the space $W^{2,2}(\X)$ and of the Hessian operator are also due to \cite{Gigli14} and are far from trivial to state. To avoid introducing lengthy technical details we omit the definition here and refer to \cite{GP20} for the details. Very roughly speaking a function is in $W^{2,2}(\X)$ provided there exists  a bilinear map $\H u: L^2(T\X)^2 \to L^0(\mm)$  satisfying a suitable integration by parts formula against functions in $\test(\X)$.  The function $|\H u|_{HS}\in L^2(\mm)$ instead denotes, formally, the Hilbert-Schmidt norm of this operator.

Integrating by parts \eqref{eq:improved bochner mms} and arguing by approximation it is possible to deduce the following Laplacian-regularity result, extending the metric setting inequality \eqref{eq:2calderon smooth}.
\begin{cor}[\hspace{1pt}{\cite[Corollary 3.3.9]{Gigli14}}]\label{thm:lapl reg}
    Let $\Xdm$ be an $\RCD(K,\infty)$ space, $K\in \rr$, and  $u\in \W(\X)$ be such that $\Delta u \in L^2(\mm).$ Then $u \in  W^{2,2}(\X)$ and 
    \begin{equation}\label{eq:2calderon nonsmooth}
        \int_{\X}  |\H u|^2\, \d \mm\le \int_{\X} (\Delta u)^2-K |D u|^2 \, \d \mm, \quad \forall u \in \dom(\Delta).
    \end{equation}
\end{cor}
Recall that if $\Omega \subset \X$ is open and convex, then  $(\overline \Omega,\sfd\restr \Omega,\mm\restr\Omega) $ is an $\RCD(K,\infty)$ space (see \cite[Theorem 7.2]{AMS16}). Hence \eqref{eq:2calderon nonsmooth} applied to $(\overline \Omega,\sfd\restr \Omega,\mm\restr\Omega) $  can be re-read as a \textit{boundary regularity} result for functions $u\in \W(\Omega)$  with Neumann boundary conditions (recall Definition \ref{def:neu}).  It it thus natural to wonder whether there exist non-convex sets in RCD spaces that admits  boundary second-order regularity estimates.

\begin{open}
    Find sufficient regularity properties (besides convexity) for  an open subset $\Omega\subsetneq\X$  of an $\RCD(K,N)$ space $\Xdm$ so that the following holds
    \[
     \int_{\Omega}  |\H u|^2\, \d \mm\le C_{\Omega} \int_{\X} (\Delta u)^2+|D u|^2 \, \d \mm, \quad \text{ for all $u\in \W_0(\Omega)$}
    \]
    (or for all $u\in \W(\Omega)$  with Neumann boundary conditions), where $C_\Omega$ is a constant depending only on $\Omega.$
\end{open}

\subsection{Lipschitz regularity}
As we mentioned in Section \ref{sec:holder} in an arbitrary locally doubling space supporting Poincaré inequality the best regularity result we can hope for harmonic functions is H\"older continuity. Nevertheless on  locally doubling $\RCD(K,\infty)$ they are Lipschitz. 
This can be seen at a formal level from  the Bochner inequality. Indeed if we assume that $\Delta u=0$ and that ${\rm Ric}_g\ge K$ we obtain that
\[
\frac12\Delta |D u|^2\ge K|D u|^2.
\]
Thus $|D u|^2$ is a subsolution and so  we can apply the Moser's iteration scheme to deduce that $|D u|^2$ is locally bounded (recall Theorem \ref{thm:sup harnack}). 

In fact a more general result holds saying that local Lipschitzianity is obtained as soon as the Laplacian is sufficiently integrable.
\begin{theorem}\label{thm:lip delta}
    Let $\Xdm$ be a locally doubling $\RCD(K,\infty)$ space and let $s>1$ satisfy \eqref{eq:d dim}. Suppose that 
    \[
    \Delta u=f, \quad \text{in $\Omega$,}
    \]
    for some $f \in L^q_{\loc}(\Omega)$ with $q>s $. Then $u$ is locally Lipschitz in $\Omega.$
\end{theorem}
The harmonic case, i.e.\ $f=0$, was obtained in \cite{Koskela-Rajala-Shanmugalingam03} in the $s$-Ahlfors regular case\footnote{A m.m.s.\ is $s$-Ahlfors regular if there exists a constant $C\ge 1$ such that $$C^{-1}r^s\le \mm(B_r(x))\le C r^s$$  for all $r\in (0,\diam(\X))$ and $x\in \X.$} and extended to all $f$ in \cite{jiangpoisson} under the same assumption. The Ahlfors regularity assumption was removed in \cite{Jiang} for $f\in L^\infty(\mm)$ and then adapted to cover all $f$ in \cite{Kell13}.

The most notable example of locally doubling $\RCD(K,\infty)$ spaces are $\RCD(K,N)$ spaces with $N<\infty$, which in particular satisfy \eqref{eq:d dim} with $s=N$ thanks to the Bishop-Gromov inequality \cite{Sturm06-2}. We recall that an (equivalent) definition of $\RCD(K,N)$ space is  the same as the one of $\RCD(K,\infty)$ given in Definition \ref{def:rcd} except that we replace \eqref{eq:weak bochner mms} with the following stronger inequality:
\[
            \int_\X |D u|^2  \Delta g\d \mm \ge \int_\X \frac{(\Delta u)^2g}{N}+\la \nabla u, \nabla \Delta u\ra g + K|D u|^2 g,
\]
(see \cite{EKS15,AMS16,CM16}).

The above Lipschitz regularity result is sharp in the following  sense: it is not possible to derive $C^1$-estimates for the gradient of harmonic functions which depend only on the lower bound of the Ricci (and even sectional) curvature. This is made rigorous in the following result obtained in \cite{DPZ19}.
\begin{theorem}
    It does not exist a modulus of continuity $\omega$ such that for all $n$-dimensional Riemannian manifolds $(M,g)$ with ${\rm Sec}_g\ge 0$ and all harmonic functions $u$ in $B_1(p)\subset M$ with $\|u\|_{W^{1,2}(B_1(p))}=1$ it holds
    \[
    ||D u|(x)-|D u|(y)|\le \omega(\sfd_g(x,y)), \quad \forall x,y \in B_{\frac12}(p).
    \]
\end{theorem}

More recently it was shown in \cite{PZZ22} that also the a version of the classical Weyl's lemma  holds in $\RCD(K,N)$ spaces. In other words the local Lipschitzianity holds for an even weaker notion of harmonic functions.
\begin{theorem}
    Let $\Xdm$ be an $\RCD(K,N)$ space with $N<\infty$ and  $\Omega \subset \X$ be open. Suppose that $u \in L^1_{\loc}(\Omega)$ satisfies 
    \[
    \int_{\Omega} u \Delta \phi =0, \quad \forall \phi \in \LIP_{bs}(\Omega)\cap \test(\X). 
    \]
    Then $u \in \LIP_{\loc}(\Omega)$.
\end{theorem}

We conclude with the following problem.
\begin{open}
   Let $\Xdm$ be an  $\RCD(K,\infty)$ space (not necessarily locally doubling) and $\Omega \subset \X$ be open. Do harmonic functions in $\Omega$ have  a continuous representative? 
\end{open}

\subsection{Eigenfunctions}
If $\Xdm$ is inf.\ Hilbertian the compactness of the embedding $W^{1,2}(\X)\hookrightarrow L^2(\mm)$ is equivalent to the Laplacian having discrete spectrum. This means that there exists a sequence of eigenvalues 
\[
\lambda_1\le \lambda_2\le .... \le \lambda_k \to +\infty
\]
counted with multiplicity and corresponding eigenfunction $\phi_k\in \W(\X)$, i.e.\ $\Delta \phi_k=-\lambda_k \phi_k$,  which form an orthogonal base of $L^2(\mm)$. It was shown in \cite{GMS15} that $\RCD(K,\infty)$ spaces with $K>0$ or bounded have discrete spectrum.

 In  $\RCD(K,N)$ spaces, $N<\infty$,  eigenfunctions  are locally Lipschitz as can be seen combining the local boundedness of solutions (recall Theorem \ref{thm:sub harnack}) with  Theorem \ref{thm:lip delta}. More precise estimates have been given \cite{AHPT21} in the compact case.
\begin{theorem}[\hspace{1pt}{\cite[Prop.\ 7.1]{AHPT21}}]
    Let $\Xdm$ be a compact $\RCD(K,N)$ space with $\diam(\X)\le D<\infty$. Then 
    \[
    \|\phi_k\|_{L^\infty(\mm)}\le C \lambda_k^\frac{N}{4}, \quad  \| |D\phi_k|\|_{L^\infty(\mm)}\le  C\lambda_k^\frac{N+2}{4}, \quad \forall \lambda_k\ge D^{-2},
    \]
    where $C$ is  a constant depending only on $K,N$ and $D.$
\end{theorem}

A version of the Weyl's asymptotic formula for eigenvalues is available in the non-collapsed case. We recall that an $\RCD(K,N)$ space $\Xdm$ is called \textit{non-collapsed} if $\mm$ is a  multiple of the $N$-dimensional Hausdorff measure in $\X$ (see \cite{GDP17} and \cite{Cheeger-Colding97I}). 
\begin{theorem}[\hspace{1pt}\cite{AHT18}]
    Let $(\X,\sfd,\mathcal H^N)$ be a compact non-collapsed $\RCD(K,N)$ space. Then 
    \[
    \lim_{k\to +\infty}\frac{k}{\lambda_k^{n/2}}=\frac{\omega_n\mathcal H^n(\X)}{(2\pi)^n}.
    \]
\end{theorem}
The same result for Dirichlet eigenfunctions was shown in  \cite{ZZweyl}.  For general $\RCD$ spaces the above asymptotic does not hold. Even more in \cite{DHPW23} examples are presented where $\lambda_k $ is asymptotic to a non-integer power of $k$ or  even has a logarithmic behavior.

We conclude mentioning a result about the number of nodal domains. Recall that given an eigenfunction $\phi_k$ its nodal domains are the connected components of the set $\{\phi_k \neq 0\}$ which is well defined by continuity of $\phi_k.$ In the smooth the classical Courant nodal domain theorem \cite{courant1923allgemeiner} states that the $k$-th eigenfunctions has at most $k$-nodal domains. A weaker version of this fact was recently obtained in the non-smooth setting. 
\begin{theorem}[\hspace{1pt}\cite{DFV24}]
    Let $(\X,\sfd,\mathcal H^N)$ be a compact non-collapsed $\RCD(K,N)$ space. Then for all $k$ sufficiently big the eigenfunction $\phi_k$ as strictly less than $k$ nodal domains.
\end{theorem}
In \cite{DFV24} actually is shown a stronger asymptotic estimate on the number of nodal domains known as Pleijel's nodal domain theorem \cite{ple}. See also \cite{CF21} for estimates related to nodal domains and nodal sets of eigenfunctions.

\subsection{Unique continuation property}
Recall the classical unique continuation property for solutions of homogeneous second order elliptic equations in the Euclidean space
\[
L u= 0, \quad \text{in $\Omega\subset \rr^n$}.
\]
\begin{itemize}
    \item[--] \textit{Weak Unique Continuation property:} if $u$ vanishes on a ball, then $u$ is identically zero.
    \item[--] \textit{Strong Unique Continuation property:} if $u$ vanishes at infinite order at some point, then $u$ is identically zero. 
\end{itemize}
A standard result in the Euclidean setting is that the strong unique continuation hold for linear operators in divergence form $L u =\div(A(x)\nabla u)$ with Lipschitz coefficients (see \cite{garofalo1986monotonicity,aronszajn1962unique}) and can be shown by using the monotonicity property of the frequency function introduced in \cite{agmon,Alm}. The regularity on the coefficient  is also sharp \cite{plis1963non,miller1974nonunique}. 

A long standing open question is the following.

\begin{open}
    Does the weak unique continuation property hold  for harmonic functions (or eigenfunctions) in $\RCD(K,N)$ spaces?
\end{open}
In \cite{deng2023failure} it was shown that at least the strong unique continuation fails.
\begin{theorem}
 For each $N\ge 4$ there exists an $\rcd(K,N)$ space $\Xdm$ and some harmonic function on some connected set $\Omega\subset \X$ that vanishes at infinite order at some $x_0\in \X$, but which is not  identically zero.    
\end{theorem}
The construction for $\X$ used in \cite{deng2023failure} is, roughly speaking, a  weighted 4-dimensional horn with tip $x_0$. In particular the space $\Xdm$ obtained is a \textit{collapsed $\RCD(K,N)$ space}.  
It is then reasonable to consider the above question  in the subclass of non-collapsed RCD spaces. This is further motivated by the fact that under this assumption the tangent space is always a metric cone (see \cite{GDP17} and \cite{Cheeger-Colding97II}). Indeed in Riemannian manifolds the unique continuation is usually obtained using the  almost monotonicity of the frequency function which ultimately is linked to how much the space looks conical (see e.g.\ \cite{colding1997large}). From a functional point of view the almost monotonicity of the frequency relies on $L^\infty$-estimates of the Hessian of the distance function from the identity (see e.g.\ \cite{mangoubi2013effect}). However for this usually upper and lower sectional curvature bounds are needed. In fact always in \cite[Theorem 2.4]{deng2023failure} it was shown that the weak unique continuation holds for metric spaces with sectional curvature bounded from above and below. On the other hand in $\RCD$ setting only $L^2$-estimates for the Hessian are available at the moment.

\section{Second order regularity for the \texorpdfstring{$p$}{}-Laplacian}\label{sec:plap}
In this section we will present some recent results  obtained in collaboration with Luca Benatti \cite{BV24} about the regularity of the $p$-Poisson equation
\[
\Delta_p u=f.
\]
In Section \ref{sec:laplacian} we observed that $W^{2,2}$-estimates for the Laplacian are obtained directly from the Bochner inequality and integration by parts. For the $p$-Laplacian we can perform a similar formal argument using instead the \textit{$p$-Bochner identity}:
given $u\in C^\infty_c(M)$,  whenever $|\nabla u|\neq 0$ it holds
\begin{equation}\label{eq:smooth p bochner}
\begin{split}
      \frac1p \div\bigg(|\nabla u|^{p-2}A(\nabla |\nabla u|^p)\bigg)&=\big|\nabla \big(|\nabla u|^{p-2} \nabla u\big)\big |^2+g(\nabla \Delta_pu,\nabla u)|\nabla u|^{p-2}\\
      & \quad \quad +{\rm Ric}_g(\nabla u,\nabla u)|\nabla u|^{2(p-2)},
\end{split}
\end{equation}
where $A\coloneqq \id +(p-2)\frac{\nabla u\otimes \nabla u}{|D u|^2}$ (see \cite[Prop.\ 3.1]{pVALTORTA}).  For $p=2$ we get back the usual Bochner-Weitzenb\"ock identity in \eqref{eq:smooth bochner}. Assuming ${\rm Ric}_g \ge K g$, integrating \eqref{eq:smooth p bochner} and then integrating by parts  term containing $\Delta_p u$  we obtain 
\begin{equation}\label{eq:pcalderon smooth}
    \int |\nabla \big(|\nabla u|^{p-2} \nabla u\big) |^2\,\d \vol_g \le \int (\Delta_p u)^2 -K|D u|^{2(p-1)}\d \vol_g.
\end{equation}
In other words we obtain the following regularity estimate:
\begin{equation}\label{eq:p-calderon smooth}
    \Delta_pu \in L^2(M)\implies |\nabla u|^{p-2} \nabla u \in W^{1,2}_C(TM).
\end{equation}
Note that $|\nabla u|^{p-2} \nabla u$ is precisely the non-linear vector field appearing in the definition of the $p$-Laplacian: $\Delta_p u=\div(|\nabla u|^{p-2} \nabla u)$. In particular \eqref{eq:pcalderon smooth} is the natural generalization  to the case $p\neq 2$ of the corresponding result for the Laplacian  \eqref{eq:2calderon smooth}. Nevertheless  results of the type in \eqref{eq:p-calderon smooth}  in the smooth setting have appeared only relatively recently. In \cite{CiaMaz18} it was showed that for $\Omega\subset \rr^n$ it holds  $\Delta_p u \in L^2_\loc(\Omega)$ with $u\in W^{1,p}_\loc(\Omega)$ if and only if $ |\nabla u|^{p-2} \nabla u \in W^{1,2}_\loc(\Omega)$.  For extensions to $p$-Laplacian type operators see  \cite{plap1,plap2,plap4,plap3}. Previously in \cite{Lou08} it was proved that $|\nabla u|^{p-1}\in \W_\loc(\Omega)$ assuming that $f\in L^{q}_\loc(\Omega)$ with $q>\max(2,n/p)$, which was generalized to the smooth Riemannian manifolds in \cite{benattithesis}.  For $p\in(1,3+\frac{2}{n-2})$ it is known that $p$-harmonic functions are in $W^{2,2}_\loc$ \cite{ManWeit88}, while it is a now standard older result that $p$-harmonic functions are $C^{1,\alpha}$ for all $p\in(1,\infty)$  \cite{dibenedetto_alphalocalregularityweak_1983,Evans82,Uh77,Ur68}.   See also the recent \cite{Sa22,LPZ23} for related results about $p$-harmonic functions in the Euclidean space.

In \cite{BV24} we obtained the following result in $\RCD$ spaces.
\begin{theorem}[Second order regularity for the $p$-Laplacian]\label{thm:main rcd inf}
		Let $\Xdm$ be a bounded $\rcd(0,\infty)$ space. Fix $p\in (1,3)$ and suppose that $u \in \dom(\Delta_p)$ with  $\Delta_p u \in L^2(\mm)$. Then $|D u|^{p-2}\nabla u\in H^{1,2}_{C}(T\X)$ and in particular $|D u|^{p-1}\in \W(\X)$. Moreover 
		\begin{equation}\label{eq:p-calderon intro}
			\int |\nabla (|D u|^{p-2}\nabla u)|^2\d \mm\le C_p \bigg(\|\Delta_pu\|_{L^2(\mm)}^2+\||D u|^{p-1}\|_{L^{1}(\mm)}\bigg),
		\end{equation}
		where $C_p>0$ is a constant depending only on $p$.
	\end{theorem}
The space  $H^{1,2}_{C}(T\X)$ denotes the space of vector fields in $L^2(T\X)$ with $L^2$-Sobolev covariant derivative. We will not give the details of the definition and refer to \cite{GP20}. Roughly speaking $v\in H^{1,2}_C(T\X)$ if there exists a bilinear map $\nabla v : L^2(T\X)^2 \to L^0(\mm)$ which satisfies a suitable integration by parts formula against test functions. The function $|\nabla v|\in L^2(\mm)$ then denotes the Hilbert-Schmidt norm of this operator.

\begin{remark}\label{rmk:now22}
    We can not expect to obtain that $u \in W^{2,2}(\X)$, as this is would be false even in the Euclidean setting, see e.g.\ \cite[Remark 2.7]{CiaMaz18}.
\end{remark}

In finite dimension  Lipschitz regularity can also be obtained.
\begin{theorem}[{Lipschitz regularity for the $p$-Laplacian, \cite{BV24}}]\label{thm:lip reg plap}
    Let $\Xdm$ be a bounded $\RCD(0,N)$ space $N<\infty$ and  fix $p\in(1,3)$. Suppose that 
    \[
    \Delta_p u=f, \quad \text{in $\X$,}
    \]
    for some $f\in L^q(\X)$ for $q>{N}$. Then $u\in \LIP(\X).$
\end{theorem}
Compare the above result with the corresponding one for the Laplacian in Theorem \ref{thm:lip delta}. Note in particular that Theorem \ref{thm:lip reg plap} applies only to globally defined functions. This is not a choice of presentation but rather a current limitation of the analysis of \cite{BV24}. On the other hand some local results are also obtained.
\begin{theorem}[{Lipschitz regularity for the $p$-electrostatic potential, \cite{BV24}}]\label{thm:lim reg electro}
    Under the assumptions of the previous theorem, let $\Omega\subset \X$ be open and $K\subset \Omega$ be compact. Suppose that  $u\in C(\overline{\Omega})$ is a solution of
     \begin{equation}\label{eq:electropp}
         \begin{cases}
             \Delta_p u=0, &\text{in $\Omega\setminus K$},\\
             u=0, & \text{in $\partial \Omega$},\\
              u=1, & \text{in $\partial K$}.\\
         \end{cases}
     \end{equation}
     Then $u$ is locally Lipschitz in $\Omega \setminus K.$
\end{theorem}
Recall from Proposition \ref{prop:electrop} that solutions to \eqref{eq:electropp} exist under mild assumptions on $\Omega$ and $K.$   In fact in \cite{BV24} it is proved that local Lipschitz regularity holds for $p$-harmonic functions  with relatively compact level sets (which is the case for the electrostatic potential).

\begin{remark}[Results for $K<0$]\label{rmk:K negative}
    In \cite{BV24} was  shown that Theorem \ref{thm:main rcd inf}, Theorem \ref{thm:lip reg plap} and Theorem \ref{thm:lim reg electro} in fact hold also in the case $K<0$, provided we suitably restrict the range of $p$.  Here we consider only the case $K=0$, since it is simpler to state and the argument is essentially the same. In the case  $N<\infty$ the range of $p$ can further be improved. See \cite[Definition 1.1]{BV24} for the precise range of $p$.
\end{remark}



\section{Proof of second order regularity estimates for the \texorpdfstring{$p$}{p}-Laplacian}\label{sec:proof}
This last section is devoted to the proof of Theorem \ref{thm:main rcd inf}.
We start with an outline of the main steps and ideas.

The core idea  to approximate solutions of $\Delta_p u=f\in L^2$ with solutions of the regularized equation:
\begin{equation}\label{eq:regularized problem}
    \Delta_{p,\eps} u_\eps\coloneqq \div((|\nabla u_\eps|^2+\eps)^\frac{p-2}{2}\nabla u_\eps)=f
\end{equation}
by sending $\eps \to 0^+$. This is a well established method  to deal with the $p$-Laplacian in the smooth setting.
We will call the $\Delta_{p,\eps}$  operator  \textit{$(\eps,p)$-Laplacian or regularized $p$-Laplacian}. The idea is that $\Delta_{p,\eps}$ is a uniformly elliptic operator and thus we can hope that the solutions $u_\eps$ are as smooth as we wish.

The argument then follows the following scheme: 
\begin{enumerate}
    \item[\textbf{Step 1}] Show that $u_\eps \to u$ in $W^{1,p}(\X)$ as $\eps \to 0^+,$
    \item[\textbf{Step 2}]Show that  $u_\eps \in W^{2,2}$,
    \item[\textbf{Step 3}] Obtain   second order estimates for $u_\eps$ which \textit{are uniform in $\eps$},
    \item[\textbf{Step 4}] Conclude the second order regularity of   $u.$
\end{enumerate}
The most difficult step, by far, is \textbf{Step 2}. This is the opposite of the smooth setting where $u_\eps \in W^{2,2}$ is an automatic consequence of standard elliptic regularity theory. However in metric setting the general elliptic regularity theory is currently limited to  H\"older estimates (see Section \ref{sec:holder}) and thus we will need to proceed in a different way.

For brevity from now we will denote 
$$\dom_0(\Delta)\coloneqq \left\{f \in \W(\X) \ : \ \Delta f \in L^2(\mm), \, \, \int_\X u \,\d \mm=0\right\}.$$
From Theorem \ref{thm:lapl reg} we have $\dom(\Delta)\subset W^{2,2}(\X).$ We recall also the following straightforward result (see\cite[Prop.\ 5.2]{BV24} for a proof).
\begin{prop}\label{lem:norms}
		Let $\Xdm$ be a bounded $\rcd(K,\infty)$ space. Then $\dom_0(\Delta)$ endowed with the norm $\|\Delta(\cdot)\|_{L^2(\mm)}$ is a Hilbert space. Moreover, the inclusion $\left(\dom_0(\Delta),\|\Delta(\cdot)\|_{L^2(\mm)}\right)\hookrightarrow \left(\dom_0(\Delta),\|\cdot\|_{\W(\X)}\right)$ is compact.
\end{prop}

\subsection{Step 1: regularization of the problem}
We start defining  the regularized operator $\Delta_{p,\eps}$.
	\begin{definition}[$(\eps,p)$-Laplacian]
		Fix $p \in (1,3)$, $\eps> 0$ and let $\Xdm$ be a bounded $\rcd(K,\infty)$ space. A function $u \in W^{1,{p-1}}(\X)$ satisfies
        $\Delta_{p,\eps} u=f$ for some $f \in L^1(\mm)$ provided
		\begin{equation}\label{eq:def eps plapl}
			\int_{\X}\la (|D u|^2+\eps)^\frac{p-2}2\nabla u,\nabla \phi\ra \d \mm=-\int_\X \phi f  \d \mm, \quad \forall \phi \in \LIP_{bs}(\X),
		\end{equation}
        in which case $f$ is unique and denoted by $\Delta_{p,\eps} u.$
	\end{definition}
We define $\Delta_{p,\eps}$ for functions in $W^{1,{p-1}}(\X)$ because we want to consider also functions in $W^{1,2}(\X)$ (note that $W^{1,2}\subset W^{1,{p-1}}(\X)$ for $p\in(1,3)$ because of  property i) in Section \ref{sec:pre}). 
    
The key result is  that solutions of $\Delta_{p,\eps} u=f$ converge to solutions of $\Delta_p u=f.$
	\begin{prop} [$\eps$-approximation of the $p$-Laplacian]\label{prop:existence of weak solutions}
		Fix $p\in(1,\infty)$ and let $p'\coloneqq \frac{p}{p-1}$. Let $\Xdm$ be a bounded $\rcd(K,\infty)$ space and   $f \in L^{p'}(\mm)$ with zero mean. Then:
\begin{enumerate}[label=\roman*)]
    \item for all $\eps>0$ there exists a unique solution $u\in W^{1,p}(\X)$ to
     \begin{equation}\label{eq:pepsi}
        \begin{cases}
            \Delta_{p,\eps} u=f,& \text{in $\X$},\\
             \int_\X u_\eps \d\mm=0.
        \end{cases}
    \end{equation} 
    \item $u_\eps \to u$ in $W^{1,p}(\X)$ as $\eps \to 0^+,$ where $u$ is the unique solution of
    \begin{equation}\label{eq:pep}
        \begin{cases}
            \Delta_{p} u=f,& \text{in $\X$},\\
             \int_\X u \d\mm=0.
        \end{cases}
    \end{equation}
\end{enumerate}
	\end{prop}
  The existence and uniqueness of solutions to both \eqref{eq:pepsi} and \eqref{eq:pep} is consequence of the direct method of calculus of variations, as discussed in Section \ref{sec:boundary} (see in particular Proposition \ref{prop:poisson neum}  and the discussion immediately before that). For the convergence part we refer to \cite[Proposition 3.3]{BV24}.

	\subsection{Step 2: Uniform a-priori estimates for regularized solutions}\label{sec:eps regularity}
	The goal of this section is to obtain second-order regularity estimates for solutions of 
	\begin{equation}\label{eq:eps poisson}
		\Delta_{p,\eps} u_\eps =f \in L^2(\mm),
	\end{equation}
    uniformly in $\eps$ and \textit{assuming that $u_\eps \in W^{2,2}(\X).$} Observe that we can not expect to obtain uniform $W^{2,2}$-estimates (recall Remark \ref{rmk:now22}). Instead we will obtain uniform Sobolev estimates on the vector field $ (|D u_\eps|^2+\eps)^\frac{p-2}{2}\nabla u_\eps.$ 

\begin{theorem}\label{thm:p calderon}\label{thm:uniform st}
     Let $\Xdm$ be an $\rcd(K,\infty)$ space and fix $u \in \dom_0(\Delta)$. 
		Then for all $\eps>0$ and  $p\in(1,3)$ we have that $v_{\eps}\coloneqq ( |D u|^2+\eps)^\frac{p-2}2\nabla u \in H^{1,2}_C(\X)$ and 
		\begin{equation}\label{eq:true unif estimate}
		      \int_{\X}|\nabla v_{\eps} |^2 +|v_{\eps}|^2\d \mm \le C_{p,K} \int_{\X} (\Delta_{p,\eps}u)^2 \d \mm + C_{p,K}\||D u|^{p-1}+\eps|D u|\|_{L^1(\mm)}.
		\end{equation}
        In particular $u_\eps \in W^{1,p}(\X)$.
\end{theorem}
We preliminary observe that formally
\begin{equation}\label{eq:developing}
    (|\nabla u|^2+\eps)^{\frac{p-2}2}\Delta_{p,\eps}u=\Delta u+(p-2)\frac{\H {u}(\nabla u, \nabla u)}{|\nabla u|^2+\eps}=f,
\end{equation}
(assuming that $u$ is sufficiently regular).
This motivates the following definition.
\begin{definition}[Developed $(\eps,p)$-Laplacian]
		Let $\Xdm$ be an $\RCD(K,\infty)$ space, $p\in(1,\infty)$ and $\eps> 0$. We defined the map $D_{\eps,p}: \dom_0 (\Delta) \to L^2(\mm)$ as
		\[
		D_{\eps,p} u\coloneqq \Delta u+(p-2) \frac{\H u( \nabla u, \nabla u)}{| \nabla u|^2+\eps} =\Delta u+(p-2)\frac{\Delta_\infty u}{|D u|^2+\eps}. \in L^2(\mm),
		\]
		 having set $\Delta_{\infty }u=|D u|\la \nabla |D u|,\nabla u \ra$.
	\end{definition}
The following result gives a rigorous version of \eqref{eq:developing}. The proof is a simple verification (see \cite[Lemma 4.2]{BV24} for the details).
\begin{lemma}\label{lem:develop}
		Fix $p \in(1,3)$ and $\eps>0.$
		Let  $\Xdm$ be a bounded $\rcd(K,\infty)$ space and let $u \in\ \dom_0(\Delta)$.   Then 
        \begin{equation}\label{eq:develop}
                \Delta_{p,\eps}u=(|D u|^2+\eps)^{\frac{p-2}{2}} D_{\eps,p}u, \quad \text{in $\X.$}
        \end{equation}
	\end{lemma}
Note that $(|D u|^2+\eps)^{\frac{p-2}{2}}\in L^2(\mm)$, as $p\in(1,3)$, hence the right hand side of \eqref{eq:develop} is in $L^1(\mm)$ and thus \eqref{eq:develop} makes sense.

\begin{proof}[Proof of Theorem \ref{thm:uniform st}]
	We prove only the case $u\in \test(\X)$. The case  $u \in \dom_0(\Delta)$ can be obtain similarly by approximation (see \cite[Lemma 4.8]{BV24}).  
    Fix an arbitrary constant $M>0$ (in the case $p<2$ we take $M=+\infty$). We will prove that
\begin{equation}\label{eq:covariant estimate N=inf}
			\int_{\X}|\nabla v_{\eps,M} |^2 \d \mm\le C_{p} \int_{\X} (\Delta_{p,\eps}u)^2  + K |v_{\eps,M}|^2  \d \mm,
		\end{equation}
		where $v_{\eps,M} \coloneqq (|D u|\wedge M)^2+\eps)^\frac{p-2}2\nabla u\in H^{1,2}_C(\X)$. We first show that \eqref{eq:covariant estimate N=inf} is enough to conclude. From \eqref{eq:covariant estimate N=inf} and by interpolation (see e.g.\ \cite[Prop.\ 2.22]{BV24}) we obtain that
        \begin{align*}
               \int_{\X}|\nabla v_{\eps,M} |^2 +|v_{\eps,M}|^2\d \mm&\le C_{p,K} \int_{\X} (\Delta_{p,\eps}u)^2 \d \mm + C_{p,K}\left(\int_\X |v_{\eps,M}|  \d \mm\right)^2\\
               &\le  C_{p,K} \int_{\X} (\Delta_{p,\eps}u)^2 \d \mm + C_{p,K}\||D u|^{p-1}+\eps|D u|\|_{L^1(\mm)}.
        \end{align*}
Hence $ v_{\eps,M}$ is bounded in $L^2(T\X)$ and clearly $|v_{\eps,M}-v_\eps|\to 0$ pointwise $\mm$-a.e.. Therefore by lower semicontinuity (see e.g.\ ) we deduce that \eqref{eq:true unif estimate} holds. From \eqref{eq:true unif estimate} for $p\ge 2$ we deduce that $|D u_{\eps}|^{p-1}\in L^2(\mm)$, which implies that $|D u_\eps|\in L^p(\mm).$ If $p\le 2$ instead $|D u_\eps|\in L^p(\mm)$ trivially. From the Poincaré inequality (e.g.\ recall \eqref{eq:neu poinc})  we deduce that also $u_\eps \in L^p(\mm)$ and so $u\in W^{1,p}(\X)$ (recall property i) in Section \ref{sec:pre}). 

We pass to the proof of \eqref{eq:covariant estimate N=inf}.
We  preliminary observe that, by definition,
\begin{align*}
    \LL_{p,\eps}(u)&\coloneqq \Delta u+(p-2)\frac{\H u(\nabla u,\nabla u)}{|D u|^2+\eps}=\Delta u+(p-2)\frac{|D u|\la\nabla |D u|,\nabla u\ra}{|D u|^2+\eps}.
\end{align*}
Hence, setting $\Delta_{\infty }u=|D u|\la \nabla |D u|,\nabla u \ra$, we have
\begin{equation}\label{eq:blabb}
    \Delta u=\LL_{p,\eps}(u)-(p-2)\frac{\Delta_{\infty }u}{|D u|^2+\eps}
\end{equation}
We  start the computation by plugging in the improved Bochner inequality \eqref{eq:improved bochner mms} the test function $\phi\coloneqq (|D u|\wedge M)^2+\eps)^{p-2}$ 
    \begin{equation}\label{eq:plug in bochner}
     -2(p-2) \int|\nabla |D u||^2  \frac{|D u|^2\phi }{|D u|^2+\eps} \d\mea\ge \int |\H u|^2\phi+\la \nabla u, \nabla \Delta u\ra \phi +K|D u|^{2} \phi \d \mea. 
    \end{equation}
Note that all integrals make sense because $|D u|\in L^\infty(\mm).$ We integrate by parts the term containing the Laplacian:
\begin{align*} 
    \int \la \nabla u, \nabla \Delta u\ra \phi \d \mm&=- \int (\Delta u)^2\phi+2(p-2)\nchi_{\{|D u|\le M\}}\frac{ \Delta u\la \nabla u, \nabla |D u|\ra|D u|\phi }{|D u|^2+\eps}\d\mm \\
    &=- \int (\Delta u)^2\phi+2(p-2)\nchi_{\{|D u|\le M\}} \frac{ \Delta u\Delta_\infty u}{|D u|^2+\eps}\phi\d \mm.
\end{align*}
We now get rid of the Laplacian terms by plugging in identity \eqref{eq:blabb}
\begin{align*}
      \int \la \nabla u,&\nabla \Delta u\ra \phi \d \mm=- \int \left(\LL_{u,\eps}(u)-(p-2) \frac{\Delta_\infty(u)}{|D u|^2+\eps}\right)^2\phi\d\mm\\
      & \quad +2(p-2) \int_{\{|D u|\le M\}} \left(\LL_{u,\eps}(u)-(p-2) \frac{\Delta_\infty(u)}{|D u|^2+\eps}\right)\frac{\Delta_\infty u}{|D u|^2+\eps}\phi\d \mm.
\end{align*}
Using twice the Young's inequality we can estimate for all $\delta>0$:
\begin{align*}
    \int &\la \nabla u, \nabla \Delta u\ra \phi \d \mm \le  \int (1+\delta^{-1}) \LL_{u,\eps}(u)^2\phi+(1+\delta)\frac{(p-2)^2(\Delta_\infty u)^2}{(|D u|^2+\eps)^2}\phi\d\mm\\
    &\quad + \int_{\{|D u|\le M\}} (2(p-2))^2\delta^{-1}  \LL_{u,\eps}(u)^2 \phi +\delta  \frac{(\Delta_\infty u)^2}{(|D u|^2+\eps)^2}\phi - \frac{2(p-2)^2(\Delta_\infty u)^2}{(|D u|^2+\eps)^2}\phi \d \mm.
\end{align*}
The above combined with the pointwise inequality  $\frac{(\Delta_\infty u)}{|D u|^2} \le |\H u|$ gives
\begin{equation}\label{eq:bochner partial est}
    \begin{split}
            \int \la \nabla u, \nabla \Delta u\ra \phi \d \mm &\le C_{p,\delta} \int \LL_{u,\eps}(u)^2\phi \d \mm + \delta C_p \int |\H u|^2 \phi \d \mm\\
         &\quad \quad +(p-2)^2\int\frac{ (1-2\nchi_{\{|D u|\le M\}})(\Delta_\infty u)^2}{(|D u|^2+\eps)^2}\phi \d \mm.
    \end{split}
\end{equation}
Finally plugging in \eqref{eq:bochner partial est} in \eqref{eq:plug in bochner}  and using Lemma \ref{lem:develop} we obtain
\begin{equation}\label{eq:final bochner formal}
    \begin{split}
        &(1-\delta C_p) \int |\H u|^2\phi\d \mm \le  C_{p,\delta} \int (\Delta_{p,\eps}(u))^2 + K|D u|^2 \phi \d\mm \\
   &+\int \underbrace{\left[ (p-2)^2\int\frac{ (1-2\nchi_{\{|D u|\le M\}})(\Delta_\infty u)^2}{(|D u|^2+\eps)^2} +2(2-p)   \frac{|\nabla |D u||^2|D u|^2 }{|D u|^2+\eps}\right]}_{R\coloneqq } \phi \d\mea.
    \end{split}
\end{equation}
Suppose first that $p\ge 2$. Then the second term of the reminder term $R$ is non-positive and can be ignored. On the other hand if also $p<3$, then $|p-2|^2<1$ and we can absorb the first term of $R$ into the left hand side, provided $\delta$ is chosen small enough, obtaining directly \eqref{eq:covariant estimate N=inf}. In the case $p<2$  it still can be proved (see \cite[Prop.\ 4.3]{BV24}) that there exists a constant $\lambda_p\in(0,1)$ such that
\[
R\le \lambda_p |\H u|^2  \quad \text{$\mm$-a.e.,}
\]
which allows to absorb the reminder term $R$ into the left-hand side obtaining again \eqref{eq:covariant estimate N=inf}. 
\end{proof}

\subsection{Step 3:  regularized solutions are in \texorpdfstring{$W^{2,2}$}{W2,2}} \label{sec:existence_regularised}
In this crucial step we show that solutions to 
\begin{equation}\label{eq:peppp}
    \Delta_{p,\eps} u_\eps=f
\end{equation}
with $f \in L^2(\mm)$ are in $W^{2,2}$. To do so we will show existence of a solution in $W^{2,2}\cap W^{1,p}$ and then exploit the uniqueness result of  Proposition \ref{prop:existence of weak solutions}. We first explain our method. Taking inspiration from identity \eqref{eq:developing} we  construct a solution in two steps:
\begin{itemize}
    \item[A)]  For all $w\in \W(\X)$ we obtain  $U_w\in \dom_0(\Delta)$ solving
    \begin{equation}\label{eq:first step}
        \Delta U_w+(p-2)\frac{\H {U_w}(\nabla w, \nabla w)}{|\nabla w|^2+\eps}=\frac{f}{(|\nabla w|^2+\eps)^{\frac{p-2}2}},
    \end{equation}
    by fixed point argument.
    \item[B)] We show that there exists a fixed point $u_\eps$ of the map $w\mapsto U_w,$ which thus will solve \eqref{eq:peppp}.
\end{itemize}

We  start with step $A).$ 
\begin{definition}[Auxiliary linear operator]\label{def:auxiliary}
		Let $\Xdm$ be an $\RCD(K,\infty)$ space, $w \in \W(\X)$ and $\eps> 0$. The map $\LL_{w,\eps}: \dom_0 (\Delta) \to L^2(\mm)$ is defined by
		\[
		\LL_{w,\eps}(U)\coloneqq \Delta U+(p-2) \frac{\H U( \nabla w, \nabla w)}{| \nabla w|^2+\eps} \in L^2(\mm),
		\]
		which makes sense because $\dom_0(\Delta)\subset W^{2,2}(\X).$
	\end{definition}
	Note that the operator  $\LL_{w,\eps}$  is linear and it is obtained by freezing the nonlinear part of the expression in \eqref{eq:developing}. Observe also that
    \begin{equation}\label{eq:compatibility}
        \LL_{u,\eps}(u)=D_{\eps,p}u,
    \end{equation}
    where $D_{\eps,p}$ was defined in the previous section.
 We can not expect to solve  precisely \eqref{eq:first step}, because the $\Delta U_w$ must have zero mean in $\X$ while the other terms together might not. Instead we will solve the slightly modified equation \eqref{eq:step 1} below.
	\begin{prop}[Existence result for $\LL_{w,\eps}$]\label{prop:step 1}
		Let $\Xdm$ be a bounded $\rcd(0,\infty)$ space and fix $p \in (1,3)$. Then for all 
		$f \in L^2\cap L^{p'}(\mm)$, $w\in \W(\X)$  and $\eps>0$ there exists (unique) $U \in \dom_0(\Delta)$ that solves
		\begin{equation}\label{eq:step 1}
			\LL_{w,\eps}(U)=g-\fint_\X(g-\LL_{w,\eps}(U))\d \mm,
		\end{equation}
        where $g\coloneqq \frac{f}{(|D w |^2+\eps)^\frac{p-2}{2}}$. 
        Moreover $U$ satisfies
		\begin{equation}\label{eq:estimates solution}
			\|\Delta U\|_{L^2(\mm)}\le \frac{ \|g\|_{L^2(\mm)}}{1-|p-2|}.
		\end{equation}
	\end{prop}
	\begin{proof}
 Define the map $$T_{f,w}: \dom_0(\Delta) \to \dom_0(\Delta)$$ given by $T_{f,w}(w)\coloneqq U,$ where $U$ is the solution to the following equation
		\begin{equation}\label{eq:regularised eps integro-diff}
			\begin{cases}
				&\Delta U=\Delta w+ g- \LL_{w,\eps}(w)- \fint_\X  (g-\LL_{w,\eps}(w)\d\mm\, \in L^2(\mm),\\
				&\int U \, \d \mm=0,
			\end{cases}
		\end{equation}
		$T_{f,w}$ is well defined. Indeed for every $h \in L^2(\mm)$ with zero mean there exists a unique function $U \in\dom(\Delta)$ such that $\Delta U=h$ and $\int U \, \d \mm=0$ (see Proposition \ref{prop:poisson neum}).

		For every $U_1,U_2\in \dom_0(\Delta)$, using that $\int (h-\fint h)^2\d \mm\le \int h^2\d \mm$ for all $h \in L^2(\mm)$, we have
		\begin{align*}
			\|\Delta T_{f,w}(U_1-U_2)\|_{L^2(\mm)}^2 &\le  \|\Delta(U_1-U_2)- \LL_{w,\eps}(U_1-U_2)\|_{L^2(\mm)}^2\\
		\text{by \eqref{eq:2calderon nonsmooth}} \quad &\le (p-2)^2 \int |\H {U_1-U_2}|^2\le  (p-2)^2  \|\Delta (U_1-U_2)\|_{L^2(\mm)}^2.
		\end{align*}
	This shows that  $T_{f,w}$  is a contraction  with respect to the norm  $\|\Delta (\,\cdot\,)\|_{L^2(\mm)}$, provided $|p-2|<1.$
		Hence there exists a fixed point $U.$
        
		Since $U$ solves \eqref{eq:step 1} we have
		\[
		\Delta U=\Delta U-\LL_{w,\eps}(U)+\LL_{w,\eps}(U)=\Delta U-\LL_{w,\eps}(U)+ g- \fint  (g-\LL_{w,\eps}(w))\d\mm,
		\]
		hence arguing as above we obtain
		\begin{align*}
			\|\Delta U\|_{L^2(\mm)}&\le \left\|\Delta U-  \LL_{w,\eps}(U)+  g\right \|_{L^2(\mm)}	\le |p-2| \||\H U|\|_{L^2(\mm)}+ \left\|g\right\|_{L^2(\mm)}\\
            \text{by \eqref{eq:2calderon nonsmooth}} \quad & \le  |p-2| \|\Delta U\|_{L^2(\mm)}+ \left\|g\right\|_{L^2(\mm)},
		\end{align*}
		 which proves \eqref{eq:estimates solution}.
	\end{proof}

We pass to step B). The goal is to apply another fixed point argument to pass from the operator $\LL_{w,\eps}$ to the true $(\eps,p)$-Laplacian.

	\begin{theorem}[Existence of regular solutions for the regularized equation]\label{thm:regularity epsilon}
		Let $\Xdm$ be a bounded $\rcd(0,\infty)$ space and let $p\in(1,3).$ Then, for every $f \in L^2\cap L^{p'}(\mm)$ with zero mean there exists a unique function $u \in  \dom_0(\Delta)\cap W^{1,p}(\X)\subset W^{2,2}(\X)$ such that
		\begin{equation}\label{eq:ep laplacian}
			\Delta_{p,\eps}u=f.
		\end{equation} 
	\end{theorem}
	\begin{proof}
    We will restrict to the case $p\ge 2 $. We will comment on the case $p<2$ at the end of the proof.
		By the scaling property of the statement we can assume that $\mm(\X)=1$.
		Fix $p\in[2,3)$ and $f\in L^2(\mm)$.
	 Fix $\eps>0$. We define the map $S_{f}: W^{1,2}_0(\X)\to \dom_0(\Delta)\subset  W^{1,2}_0(\X)$ given by $S_{f}(w)\coloneqq u$, where $u\in \dom_0(\Delta)$ is the (unique) solution to
		\begin{equation}\label{eq:pre Schauder fixed}
			\begin{cases}
				&	\LL_{w ,\eps}u=h(\eps, w )-\int (h(\eps, w )-	\LL_{\nabla w ,\eps}u)\d \mm, \\
				&\int u \, \d \mm=0,
			\end{cases}
		\end{equation}
		where $h(\eps,w)\coloneqq 	\frac{f}{\left (|\nabla w|^2+\eps \right)^{\frac{p-2}{2}}}\in L^2(\mm)$. 
		The map $S_{f}$ is well defined, since \eqref{eq:pre Schauder fixed} admits a unique solution $u$ by Proposition \ref{prop:step 1}. We also stress that $S_f$ is non-linear. 
        
        We aim to apply the Schauder fixed point theorem (see e.g.\ \cite[Corollary 11.2]{GilbargTrudinger}). We need to show that $S_f$ is continuous with relatively compact image.
		
		\smallskip
		
		\noindent \textsc{Precompact image.} The set $S_f(W^{1,2}_0(\X))$ is relatively compact in $W_0^{1,2}(\X).$  Indeed, from \eqref{eq:estimates solution} we have
		\begin{equation}\label{eq:a priori bounds}
			\|\Delta S_f(w)\|_{L^2(\mm)}\le  C_p\|h(\eps,w)\|_{L^2(\mm)}\le  
				C_p \eps^{\frac{2-p}{2}}\|f\|_{L^2(\mm)}
		\end{equation}
		for every $w \in W_0^{1,2}(\X)$. Moreover, the inclusion $\dom_0(\Delta)\hookrightarrow \W_0(\X)$ is compact by Proposition \ref{lem:norms}. 
		
		\noindent \textsc{Continuity.} The map $S_f$ is continuous in $W^{1,2}(\X)$. It is enough to prove that  if $w_n \to w$ in $\W(\X)$ then $\|\Delta(S_f(w_n)-S_f(w))\|_{L^2(\mm)}\to 0.$ Indeed the inclusion $(\dom_0(\Delta),\|\Delta(\cdot)\|_{L^2(\mm)})$ in $W^{1,2}(\X)$ is compact  and thus continuous. It is sufficient to show that this holds for every subsequence, up to a further subsequence. Fix then a non-relabeled subsequence.
		Up to extracting a further non-relabeled subsequence, we can then assume that $|\nabla w_n-\nabla w|\to 0$ and $|\nabla w_n|\to |\nabla w|$ in $\alme.$ For ease of notation we write  $h_n\coloneqq h(\eps,w_n),h\coloneqq h(\eps,w)$,  $\LL_n\coloneqq\LL_{ w_n,\eps}$, $\LL\coloneqq\LL_{ w ,\eps}$, $u_n\coloneqq S_f(w_n)$ and $u\coloneqq S_f(w).$ 
 Therefore using the equation \eqref{eq:pre Schauder fixed} and the inequality above we can compute
		\begin{align*}
			\|\Delta& (u_n-u)\|_{L^2(\mm)}\\
			&= \|\Delta(u_n-u)+  ( h_n-\LL_n(u_n))+ (\LL(u)- h)-
			\int_\X  ( h_n-\LL_n(u_n))+ (\LL(u)- h) \|_{L^2(\mm)}\\
			&\le \|\Delta(u_n-u)-  \LL_n(u_n)+ h_n+ \LL(u)- h\|_{L^2(\mm)}\\
			&\le \|\Delta(u_n-u)-\LL_n(u_n)+ \LL(u)\|_{L^2(\mm)}+\| h_n- h\|_{L^2(\mm)}\\
			&\le \begin{multlined}[t]\|\Delta(u_n-u)-  \LL_n(u_n-u)\|_{L^2(\mm)}+\| (\LL(u)-\LL_n(u))\|    +\|(h_n-h)\|_{L^2(\mm)}\end{multlined}\\
			& \le\begin{multlined}[t] |p-2| \|\Delta(u_n-u)\|_{L^2(\mm)}+C_N\|\LL(u)-\LL_n(u)\|    +\|h_n-h\|_{L^2(\mm)},\end{multlined}
		\end{align*}
		 Since $ |p-2|<1$, we can absorb the first term of the right-hand side into the left-hand side. Hence, it is sufficient to show that all the other terms on the right-hand side go to zero as $n \to +\infty$.
		From the definition of the operators $\LL_n,\LL$ it is easily checked that
		\begin{align*}
			\left |{\LL_{n}(u)-\LL(u)}\right|\to 0, \quad  \text{$\mea$-a.e.}, \quad \left |{\LL_n(u)-\LL(u)}\right|\le 2|p-2||\H u|.
		\end{align*}
		Hence from the dominated convergence theorem we deduce that $\|\LL_{n}(u)-\LL(u)\|_{L^2(\mm)}\to 0$. The argument to show that $\|h_n-h\|_{L^2(\mm)}\to 0$ is similar.
		\smallskip
		
		Hence  we deduce that $S_f$ has a fixed point $u\in \dom_0(\Delta)$ which then satisfies
		\begin{equation}\label{eq:almost pde}
			\LL_{u,\eps}(u)=h(\eps,u)-\underbrace{\int(h(\eps,u)-	D_{\eps,p}(u))\d \mm}_{\lambda_u\coloneqq }.
		\end{equation}
		\smallskip
	From Lemma \ref{lem:develop} and  \eqref{eq:almost pde} (recalling also \eqref{eq:compatibility}) we deduce that
		\[
		\Delta_{p,\eps} u= f- (|D u|^2+\eps )^{\frac{p-2}{2}}\lambda_u, \quad \text{in $\X$.}
		\] 
		Since both $\Delta_{p,\eps} u$ and $f$ has zero mean we deduce that $\lambda_u=0$ and thus \eqref{eq:ep laplacian}. Finally the fact that $u\in W^{1,p}(\X)$ follows from Theorem \ref{thm:uniform st} while the uniqueness from Proposition \ref{prop:existence of weak solutions}. 
		
In the case $p<2$ the core of the argument is the  same only that the function  $h(\eps,w)=  (|\nabla w|^2+\eps)^{\frac{2-p}{2}}$ is not necessarily in $L^2(\mm)$. Hence we need first to perform a cut-off and consider the function $h(\eps,w)\coloneqq ((|\nabla w|\wedge M)^2+\eps)^{\frac{2-p}{2}}$ and then send $M\to +\infty$ (see \cite[Theorem 5.6]{BV24} for the details).
\end{proof}

	\subsection{Step 4: going back to the \texorpdfstring{$p$}{p}-Laplacian and conclusion}
	\label{sec: main results}
	
	We can finally pass to the limit sending $\eps\to 0^+$ combining the results of the previous sections and obtaining the regularity result for the $p$-Laplacian.
	\begin{proof}[Proof of Theorem \ref{thm:main rcd inf}]
		We  will prove the result for $f \in L^2\cap L^{p'}(\mm)$. From this the general case can be obtained with an additional approximation procedure (see the proof of \cite[Theorem 6.1]{BV24}). Fix a sequence $\eps_n>0$ such that  $\eps_n\to 0$.  Note that $f$ must have zero mean.  By Theorem \ref{thm:regularity epsilon}  for every $n \in\nn$ there exists a unique function $u_n \in  \dom_0(\Delta)\cap W^{1,p}(\X)$ such that $\Delta_{p,\eps_n}u_n=f$. Moreover, by Proposition \ref{prop:existence of weak solutions} it holds $u_n\to u$ in $W^{1,p}(\X)$. In particular, setting $v_n\coloneqq (|D u_n|^2+\eps)^\frac{p-2}2\nabla u_n$ and $v\coloneqq |D u|^{p-2}\nabla u$ we have $|v_n-v|\to 0$ $\mea$-a.e.. Moreover from Theorem \ref{thm:uniform st} we have that $v_n\in H^{1,2}_C(T\X)$ and 
        \begin{equation}\label{eq:est vn}
           \int_\X |\nabla v_n |^2+|v_n|^2 \d \mm\le C_{p,K} \int f^2\d \mm+ C_{p,K}\||D u_n|^{p-1}+\eps_n|D u_n|\|_{L^1(\mm)}.
        \end{equation}
        Since $u_n$ is converging in $W^{1,p}(\X)$ we deduce that $v_n$ is bounded in $H^{1,2}_C(\mm).$
        In particular $v_n\to v$ in $L^2(T\X)$ and by the lower semicontinuity of the energy  we obtain $v\in H^{1,2}_C(T\X)$ together with \eqref{eq:p-calderon intro}.
	\end{proof}

		\def\cprime{$'$}

\end{document}